\documentclass{amsart}
\usepackage{amsmath,amssymb,amsthm,latexsym}
\usepackage{graphicx}
\usepackage{stmaryrd}		% \llbracket and \rrbracket
\usepackage{xcolor}
\usepackage{hyperref}
\usepackage{tikz}

\usepackage{amssymb}
\usepackage{amsthm}
\usepackage{thmtools}

\usetikzlibrary{arrows,automata}

\usepackage{pgf}
\usepackage{tikz}
\usetikzlibrary{arrows,automata}
\usepackage[latin1]{inputenc}
\usepackage{verbatim}

\usepackage{color}

\newcommand{\thick}{\mbox{\rm thk}}
\renewcommand{\th}[2]{\theta^{(#1)}_{#2}}
\newcommand{\ho}[1]{^{(#1)}}
\newcommand{\Aa}{{\mathcal A}}

\newcommand{\pmax}{\pi}

\newcommand{\pmi}{\pmax^{-1}}

\newcommand{\im}{\mbox{\rm im}}
\numberwithin{equation}{section}
\numberwithin{figure}{section}

\newtheorem{thm}{Theorem}[section]
\newtheorem{cor}[thm]{Corollary}
\newtheorem{lem}[thm]{Lemma}
\newtheorem{prop}[thm]{Proposition}

\theoremstyle{definition}
\newtheorem{definition}[thm]{Definition}

\newtheorem{remark}[thm]{Remark}

\newtheorem{myexp}{Example}

\newcounter{paranum}[section]

\newcommand{\mc}{\mathcal}
\newcommand{\w}{\omega}

\renewcommand{\:}{\colon}

\newcommand{\om}{\omega}
\newcommand{\db}[1]{\bar{#1}}
\renewcommand{\subset}{\ssq}

\newcommand{\ssq}{\ensuremath{\subseteq}}

\newcommand{\eps}{\ensuremath{\varepsilon}}

\newcommand{\inte}{\ensuremath{\mathrm{int}}}

\newcommand{\alphlist}{\begin{list}{(\alph{enumi})}{\usecounter{enumi}\setlength{\parsep}{2pt}
      \setlength{\itemsep}{1pt} \setlength{\topsep}{5pt}
      \setlength{\partopsep}{3pt}}}
\newcommand{\arablist}{\begin{list}{(\arabic{enumi})}{\usecounter{enumi}\setlength{\parsep}{2pt}
          \setlength{\itemsep}{1pt} \setlength{\topsep}{5pt}
          \setlength{\partopsep}{3pt}}}
\newcommand{\romanlist}{\begin{list}{(\roman{enumi})}{\usecounter{enumi}\setlength{\parsep}{2pt}
              \setlength{\itemsep}{1pt} \setlength{\topsep}{5pt}
              \setlength{\partopsep}{3pt}}}
\newcommand{\Romanlist}{\begin{list}{(\Roman{enumi})}{\usecounter{enumi}\setlength{\parsep}{2pt}
              \setlength{\itemsep}{1pt} \setlength{\topsep}{5pt}
              \setlength{\partopsep}{3pt}}}
\newcommand{\bulletlist}{\begin{list}{$\bullet$}{\setlength{\parsep}{2pt}
                \setlength{\itemsep}{1pt} \setlength{\topsep}{5pt}
                \setlength{\partopsep}{3pt}\setlength{\leftmargin}{15pt}}} 
\newcommand{\Alphlist}{\begin{list}{(\Alph{enumi})}{\usecounter{enumi}\setlength{\parsep}{2pt}
      \setlength{\itemsep}{1pt} \setlength{\topsep}{5pt}
      \setlength{\partopsep}{3pt}}}
 \newcommand{\listend}{\end{list}}

\newcommand{\N}{\ensuremath{\mathbb{N}}} 
\newcommand{\R}{\ensuremath{\mathbb{R}}}
\newcommand{\Z}{\ensuremath{\mathbb{Z}}}

\begin{document}

\title{Tame or wild Toeplitz shifts} \author{}

\pagestyle{plain}

\author{Gabriel Fuhrmann}
\address{Department of Mathematics, Imperial College London,  United Kingdom.
	}
\email{gabriel.fuhrmann@imperial.ac.uk}

\author{Johannes Kellendonk}
\address{
	 Institut Camille Jordan, Universit\'{e} Lyon~1, France
	}
\email{kellendonk@math.univ-lyon1.fr}

\author{Reem Yassawi}
\address{ School of Mathematics and Statistics, The Open University,  United Kingdom.
}
\email{reem.yassawi@open.ac.uk}

%\marginpar{actually used 2020 subject classification}

\date{\today}

\begin{abstract} 
We investigate tameness of Toeplitz shifts. 
By introducing the notion of extended Bratteli-Vershik diagrams, 
we show that such shifts with finite Toeplitz rank 
are tame if and only if there are at most countably many orbits of singular fibres over the maximal equicontinuous factor.
The ideas are illustrated using the class of substitution shifts.
A body of elaborate examples shows that the assumptions of our results cannot be relaxed.
\end{abstract}

%%%%%%%%%%%%%%%%%%%%%%%%%%%%%%%%%%%%%%%%%%%%%%%%%%%%%%%%%%%%%%%%%%%%%%%%%%%%%%%%%%%%%%%%%%%%%%%%%%%%%%%%%%%%%%%%%%%%%%%%%%%

 \maketitle

\section{Introduction}

Tameness is known for its many facets, 
related  to deep theorems in topology, Banach space theory and model theory, such as Rosenthal's $\ell^1$-embedding theorem or the BFT-dichotomy.
It was introduced to dynamics by K\"ohler, under the name \emph{regularity} \cite{Kohler1995}, and the joint efforts
of the community helped to shed light on general structural properties of tame dynamical systems
\cite{Glasner2006,GlasnerMegrelishvili2006,KerrLi2007,Huang2006,Glasner2018,FuhrmannGlasnerJagerOertel2018}, see
also \cite{GlasnerMegrelishvili2018} for an exposition and numerous further references. The opposite of tame is wild, but with mathematical modesty and to not overuse the word wild, the community simply calls wild dynamical systems  non-tame. 
One fascinating phenomenon is that non-tame systems are not just not tame,
but qualitatively very far from tame systems.
This is most visibly reflected in the following characterisation: 
A $\Z$-action defined by a homeomorphism $T$ on a compact metrisable space $X$ is {\em tame} if 
the cardinality of its Ellis semigroup $E(X,T)$ is at most that of the continuum $c$, {whereas if it is non-tame, 
$E(X,T)$ must contain a copy of $\beta \N$.}

To put our results in context, it is important to say a few words about the relation between 
tameness and almost automorphy---a classical notion extensively studied by Veech \cite{Veech1965}.
Huang realised that tame minimal $\Z$-actions are
necessarily uniquely ergodic and \emph{almost automorphic}
\cite{Huang2006}, that is, the maximal equicontinuous factor $\mathcal Z$ contains a point which has 
a unique pre-image under the factor map $\pi:X\to \mc Z$. 
Such a point is called {\em regular} and correspondingly, points $z\in \mc Z$ for  which $\pi^{-1}(z)$ consists of  more than one element are called {\em singular}.
Recently, Glasner extended Huang's result to minimal actions of general groups possessing
an invariant measure \cite{Glasner2018}.
Based on these results, it was proved in \cite{FuhrmannGlasnerJagerOertel2018} 
that tame minimal systems are actually regularly almost automorphic, which means that the set of regular points has full (Haar) measure.

The vanishing, in measure, of the singular points  is thus a necessary condition for tameness, but it is far from being sufficient. For shift dynamical systems, the set of singular points is the union of the $\Z$-orbits of its {\em discontinuity points}: these are  the points in the maximal equicontinuous factor whose fibre contains two points which disagree on their $0$-coordinate. 
A binary almost automorphic shift whose maximal equicontinuous factor is an irrational rotation on the circle is non-tame if its set of discontinuity points
is a Cantor set \cite[Proposition~3.3]{FuhrmannGlasnerJagerOertel2018}.
Note that among such systems, one easily finds examples for which the set of singular points is a  zero measure set.
Furthermore, it is possible to construct non-tame almost automorphic systems 
for which the set of singular points consists of a
single orbit \cite{FuhrmannKwietniak2020}. 
However, in this case, the pre-image of a singular point under the factor map has to be uncountable.
These results suggest that non-tameness is related to 
the smallness of the difference between $X$ and its maximal equicontinuous factor, where smallness is computed either via a measure or via cardinality, either in $X$ or its maximal equicontinuous factor.
Our results largely affirm this suggestion but emphasize that this relation is generally speaking more subtle.

We investigate the notion of tameness for the class of Toeplitz shifts,  which are almost automorphic extensions of
 odometers (procyclic group rotations). 
In this class, Oxtoby found a first example of a minimal system which is not uniquely ergodic \cite{Oxtoby:1952}.
 Jacobs and Keane defined and studied Toeplitz shifts systematically in \cite{Jacobs-Keane}, recognising their close relation to Toeplitz's constructions in \cite{Toeplitz:1928}. Toeplitz shifts  have  since enjoyed ample attention due to their dynamical diversity
and their relevance in measurable and topological dynamics, see e.g.\ \cite{MarkleyPaul1979,Williams,Downarowicz1991,BulatekKwiatowski1992,Iwanik1996,DownarowiczLacroix1998,GjerdeJohansen,DownarowiczSerafin2003,DownarowiczKasjan2015,BaakeJaegerLenz2016} as well as \cite{Downarowicz2005} for a survey and further references.   
Gjerde and Johansen \cite{GjerdeJohansen}  represent Toeplitz shifts as   Bratteli-Vershik systems where the Bartlett diagrams are appropriately constrained, and it is this representation that is particularly useful to us.
The Toeplitz shift has {\em  finite Toeplitz rank} if it has such a representation for which  the number of vertices at each level  of the Bartlett diagram is uniformly bounded. 
 We show

\begin{thm}\label{thm: A}
Let $(X, T)$ be a Toeplitz shift of finite Toeplitz rank.
Then $(X,T)$ is tame if and only if its maximal
equicontinuous factor has only countably many singular points.
\end{thm}

As a corollary to Theorem~\ref{thm: A}, we get the following necessary and sufficient criterion for tameness of substitution shifts of constant length. 
Note that in this case, the property of having only finitely many orbits of singular points can be easily read off from an associated graph
introduced in \cite{CovenQuasYassawi2016}.
We elaborate on this in the main body of this work.
\begin{thm}\label{thm: B}
Let $\theta$ be a primitive aperiodic substitution of constant length. 
The associated shift $(X_\theta, T)$ is tame if and only if 
$\theta$ has a coincidence and the maximal equicontinuous factor contains only finitely many orbits of singular points.
\end{thm}

Considering the statement of Theorem \ref{thm: A}, it would be tempting to guess that 
the presence of uncountably many singular fibres implies non-tameness. But
it turns out that in spite of Theorem~\ref{thm: A}, neither the cardinality of the singular points nor that of the fibres decide whether $(X,T)$ is tame or otherwise. We show
\begin{thm}\label{thm: C}
 There is a tame binary Toeplitz shift whose set of discontinuity points 
  is a Cantor set and so its maximal equicontinuous factor has uncountably many singular points. Moreover, there exist tame as well as non-tame binary Toeplitz shifts 
 with a unique singular orbit whose fibres are uncountable.
\end{thm}

The first example shows that the aforementioned sufficient criterion for non-tameness of almost automorphic shifts over irrational rotations of the circle, namely that the set of {\em discontinuity points} 
is a Cantor set, does not generalise when we replace the circle by a totally disconnected set.
The proof of the second part of Theorem \ref{thm: C} is based on a slightly refined version
of the constructions carried out in \cite{FuhrmannKwietniak2020}, where the possible non-tameness of systems with a unique singular
orbit was already observed.
Somewhat surprisingly, our constructions allow us to
deduce the fact that it is possible for a 
minimal system $(X,T)$ to be non-tame even it is forward tame, that is, 
the corresponding forward motion is tame (see Section~\ref{sec: basic notation} for definitions).

In light of Theorem \ref{thm: C}, it is not straightforward to identify a property that implies  non-tameness for Toeplitz shifts. 
Whilst  stipulating that a Cantor set of singular fibres exists is not sufficient, additionally controlling the points in such a set of fibres does the trick. 
This control is granted by a property we refer to as {\em thickness}.
Its definition necessitates extending the Bratteli Vershik diagrams associated to Toeplitz shifts, see Definitions \ref{extended-diagram-definition}  and \ref{def:thick}.
Our main result, from which
Theorems \ref{thm: A} and \ref{thm: B} follow, is

\begin{thm}\label{thm:0}
Every thick Toeplitz shift
 is non-tame.
\end{thm}

Let us close the introduction with pointing out an interesting connection
between tameness of substitution shifts and \emph{amorphic complexity}, a topological invariant which was introduced to detect complex behaviour in the zero entropy regime \cite{FuhrmannGroegerJaeger2016}.
In the case of symbolic systems, it coincides with the box dimension
of the maximal equicontinuous factor, the box dimension being determined through an averaging metric which is defined by the dynamics.
For  constant length substitution  shifts on a binary alphabet,
Theorem~\ref{thm: B} and \cite[Theorem~1.1]{FuhrmannGroger2019} yield that 
$(X_\theta,T)$ is tame if and only if its amorphic complexity is $1$.

The work is organised as follows. 
In the remainder of this section, we collect some basic notions which are needed throughout the article.
In Section \ref{odometer-toeplitz-substitution} we give the Bratteli-Vershik and Toeplitz background needed and prove Theorem \ref{thm:0}. 
We take some time to explain our results specified to the family of substitution shifts as this is an important family where the proofs are simpler and motivational for the following notions.
We recall the general construction of semicocycle extensions 
in Section~\ref{sec: semicocycle extensions} and 
we obtain criteria for tameness of almost automorphic shifts by investigating their discontinuity points and their formulation as semicocycle extensions
in Section~\ref{sec: criteria for tameness of almost automorphic shifts}.
While our examples and main results solely deal with symbolic shifts on finite alphabets,
it turns out that the extra effort due to treating general almost automorphic systems in that section is almost negligible.
Finally, the first part of Theorem~\ref{thm: C} is proven in Section~\ref{sec: a tame toeplitz shift with uncountably many singular fibres} and the second part is dealt with in the last section.

\medskip
\noindent{\bf Acknowledgements.} 
This project has received funding from the European Union's Horizon 2020
research and innovation programme under the Marie Sk\l{}odowska-Curie grant
agreement No 750865.
 
\subsection{Basic notions and notation}\label{sec: basic notation}
Most of this section is standard, see e.g.\ \cite{GottschalkHedlund1955,Auslander1988,deVries1993}.
We provide additional references for less standard material in the text.

We denote by $\N$ the positive integers and by
$\N_0$ the nonnegative integers.
A \emph{dynamical  system} is a continuous 
$\Z$-action (or $\N$-action) on a compact metric space $X$.
Such a system is specified by a pair $(X,T)$ where $T$ is a continuous self-map on $X$.
Clearly, $T$ is invertible when dealing with a $\Z$-action. 
Further, every $\Z$-action restricts to two $\N$ actions, its so-called \emph{forward motion} given by positive powers of $T$, and its \emph{backward motion} 
given by negative powers of $T$. 
Notions such as \emph{subsystem, minimality, (topological) factor, extension, conjugacy} etc.\ have their standard meaning.

A $\Z$-action is called \emph{equicontinuous} if the family $\{T^n : n\in\Z\}$ is equicontinuous. 
It is well-known, see e.g. \cite{Downarowicz2005},  that a minimal equicontinuous system $(\mc Z,S)$ is a 
{\em minimal rotation}, that is, there is a continuous abelian group structure on $\mc Z$ and an element $g\in \mc Z$ such that the homeomorphism $S$ is given by adding $g$, $S(z) = z+g$. 
We hence refer to such a system by $(\mc Z,+g)$.
Minimality implies that $g$ is a \emph{topological generator} of $\mc Z$, that is, $\{ng : n\in\Z\}$ 
is dense in $\mc Z$.

An equicontinuous factor of $(X,T)$ is {\em maximal} if any other equicontinuous factor of $(X,T)$ factors through it.  $(X,T)$ is an \emph{almost one-to-one extension} of a system $(Y, S)$
if the associated factor map $\pi\: X\to Y$ is {\em almost one-to-one}, that is, if $\{x\in X: \pi^{-1}(\{\pi(x) \})=\{x\}\}$ is dense in $X$. 
A system $(X,T)$ is {\em almost automorphic} if it is an almost one-to-one extension of 
a minimal equicontinuous factor. 
Almost automorphic systems are necessarily minimal and for minimal systems $\{x\in X: \pi^{-1}(\{\pi(x) \})=\{x\}\}$ is a dense $G_\delta$ if it is non-empty. 
Given an almost automorphic system $\pi:(X,T)\to (\mc Z,+g)$,
we call the points $z\in \mc Z$ with a unique $\pi$-preimage {\em regular}, and those which have multiple preimages {\em singular}. 
Correspondingly, we call a $\pi$-fibre 
$\pi^{-1}(z)$ \emph{regular} if it is a singleton and otherwise \emph{singular}. 
A point $x\in X$ in a regular fibre is also called an {\em injectivity point}. 

In Section~\ref{sec: semicocycle extensions}, we discuss a natural representation of
almost automorphic $\Z$-actions as (bilateral) shifts. 
Here by \emph{shift} we mean a subsystem of the system $(K^\Z,\sigma)$ given by the set of $K$-valued bilateral sequences $(x_n)_{n\in \Z}\in K^\Z$, equipped with the product topology, where $K$ is a compact metric space, and $\sigma$ is the left shift, $\sigma(x)_n = x_{n+1}$. 

The theory of topological independence allows for an alternative characterisation
of \mbox{(non-)}tameness 
which turns out to be particularly convenient in explicit computations and
doesn't explicitly involve the Ellis semigroup \cite{KerrLi2007}.
Given a system $(X,T)$ and subsets $A_0,A_1 \ssq X$,
we say that $J\ssq \Z$ is an \emph{independence set} for $(A_0,A_1)$ if
for each finite subset $I\ssq J$ and every choice function $\varphi \in \{0,1\}^I$ there exists $x\in X$ such that
$T^i(x)\in A_{\varphi(i)}$ for each $i\in I$.
By combining the results from \cite{KerrLi2007} and the aforementioned shift representation of
almost automorphic systems, we obtain the following characterisation of non-tameness 
which we actually understand as its definition in the following.
\begin{prop}[{cf. \cite[Proposition~6.4]{KerrLi2007} \& \cite[Proposition~3.1]{FuhrmannKwietniak2020}}]\label{prop: independence implies non-tame} 
 A shift $(X,\sigma)\subset (K^\Z,\sigma)$
  is non-tame (or wild) if and only if there are 
 disjoint compact sets $V_0,\, V_1\ssq K$ and a sequence
 $(t_n)_{n\in \N}$ such that for each choice function $\varphi \in \{0,1\}^\N$ there is  
 $(x_n)_n\in X$ for which $x_{t_n}\in V_{\varphi(n)}$ for all $n\in \N$.
\end{prop}
We call the infinite sequence $(t_n)$ of the proposition an {\em independence sequence} for the pair $(V_0,V_1)$ and note that its elements must be pairwise distinct. In the terminology of 
\cite{KerrLi2007} the elements of $(t_n)$ form an independence set for the pair of cylinder sets 
\[
[V_i]=\{x\in X\: x_0\in V_i\} \qquad (i=0,1).
\]
In line with the above characterisation of non-tameness, we call a shift $(X,\sigma)$
\emph{forward non-tame} if it allows for an infinite independence sequence (for some disjoint cylinder sets)
of positive integers.
Similar to the characterisation of non-tameness of $\Z$-actions given in the introduction, this is the case if and only if the Ellis semigroup of the $\N$-action given by the 
forward motion of $(X,\sigma)$ contains a copy of $\beta \N$.

\begin{remark}\label{rem: tameness of shifts with finite alphabet}
In the case of symbolic shifts, i.e.\ when $K=\{a_0,\ldots,a_n\}$ is finite, it
is straightforward to see that $V_0$ and $V_1$ can be chosen to be singletons $\{a_k\}$ and $\{a_\ell\}$, respectively,
and we simply write $[a_k],[a_\ell]$ for the corresponding cylinder sets.
In particular, in order to prove tameness (or forward tameness) of binary shifts, i.e. where $K = \{a,b\}$, it suffices to show that $([a],[b])$ doesn't allow for
an infinite independence set (of positive integers).
\end{remark}

\section{Toeplitz shifts}\label{odometer-toeplitz-substitution}
Our main goal in this work is to characterise those Toeplitz shifts which are tame. 
In the first two parts of this section, we define Toeplitz shifts, and in particular a simple and ubiquitous subclass, 
which is defined by (primitive, aperiodic) substitutions of constant length which have a coincidence. 
While our results are far more general, we will later use examples from this class to illustrate our constructions.
These constructions are carried out in Section~\ref{extended-BV}, where we turn to the Bratteli-Vershik representation of Toeplitz shifts to  prove Theorems~\ref{thm: A}, \ref{thm: B} and \ref{thm:0}.

\subsection{Odometers and Toeplitz shifts}
Given a sequence $(\ell_n)$ of natural numbers,
we work with the group \[\Z_{(\ell_n)} :=\prod_{n} \Z/\ell_n\Z,\] where the group operation is given by coordinate-wise addition with carry.
For a detailed exposition of equivalent definitions of $\Z_{(\ell_n)}$, we refer the reader to \cite{Downarowicz2005}. 
Endowed with the product topology over the discrete topology on each $\Z/\ell_n\Z$, the group $\Z_{(\ell_n)}$ is a compact metrizable  topological group, 
where the unit $z=\dots 001$, which we simply write as $z=1$, is a topological generator.
We write elements $(z_n)$ of $\Z_{(\ell_n)}$ as left-infinite sequences $\dots z_2z_1$ where $z_n\in \Z/\ell_n\Z$,
so that addition in $\Z_{(\ell_n)}$ has the carries 
propagating to the left as usual in $\Z$. If $\ell_n=\ell$ is constant, then  $\Z_{(\ell_n)}= \Z_\ell$ is the classical ring of $\ell$-adic integers.

With the above notation, an {\em odometer} is a dynamical system $(\mc Z, +1)$ where 
 $\mc Z=\Z_{(\ell_n)}$ for some sequence $(\ell_n)$. 
A {\em  Toeplitz shift} is a symbolic shift $(X,\sigma)$, $X\subset \mathcal A^{\Z}$ with $\mathcal A$ finite, which is an almost automorphic extension of an odometer and hence minimal. 

\subsection{Constant length substitutions}\label{sec: substitutions}
A special class of almost automorphic extensions of odometers is the class of 
{\em primitive aperiodic constant length substitutions which possess a coincidence.} 
Since we illustrate our theory mainly with examples from this class, and because the proof of our main result simplifies for this class,
we give the reader both a brief exposition of substitutions, and   also a flavour of our main result for this important class of almost automorphic extensions.

Let $\Aa$ be a finite set, referred to as \emph{alphabet}.
A  {\em substitution of (constant) length $\ell$} over $\Aa$ is a map $\theta:\mathcal A\rightarrow \mathcal A^{\ell}.$
We can write such a substitution as follows: there are $\ell$ maps $\theta_i:\mathcal A \rightarrow \mathcal A$, $0\leq i \leq \ell-1$ such that $\theta(a) = \theta_0(a)\mid \ldots \mid \theta_{\ell-1}(a)$
for all $a\in\mathcal A$, where $\mid$ is to separate the concatenated letters.

 We use concatenation to 
extend $\theta$ to a map on finite and infinite words in $\mathcal A$. 
We say that $\theta$ is {\em primitive} if there is some 
$k\in \N$ such that for any $a,a'\in \mathcal A$,
the word $\theta^k(a)$ contains at least one  occurrence of $a'$. 
We say that a finite word is {\em allowed} for $\theta$ if it appears as a subword in some $\theta^k(a)$, $a\in \mathcal A, k\in\N$.

Let $X_\theta\subset \Aa^\Z$ be the set of  bi-infinite sequences all of whose finite subwords are allowed for $\theta$. Then  $( X_\theta,  \sigma)$ is the  {\em substitution shift}  defined by $\theta$.
Primitivity  of $\theta$ implies that $(X_\theta,\sigma)$ is minimal.
We say that a primitive substitution is {\em aperiodic} if $X_\theta$ is not comprised of $\sigma$-periodic sequences. 
This is the case if and only if $X_\theta$ is an infinite space. 

The shift $(X_\theta,\sigma)$ of 
a primitive aperiodic substitution of constant length $\ell$ factors onto the odometer $(\Z_\ell,+1)$. Indeed, for any $n\geq 1$ the space $X_\theta$ can be partitioned into $\ell^n$ clopen subsets $\sigma^i(\theta^n(X_\theta))$, $i=0,\cdots,\ell^n$, and the factor map is given by
$X_\theta \ni x\mapsto \dots z_2 z_1\in \Z_\ell$ where $z_n$ is the unique $i$ such that $x\in \sigma^i(\theta^n(X_\theta))$ \cite{Dekking1977}. The maximal equicontinuous factor of $(X_\theta,\sigma)$ is therefore a covering of $(\Z_\ell,+1)$ and the degree of this covering is called the {\em height} of the substitution. The height $h$ is always finite. 
Given $\theta$, there is a primitive aperiodic substitution $\theta'$, referred to
as the \emph{pure base} of $\theta$, which is of
the same length $\ell$, has height $1$ and is such that $(X_\theta,\sigma)$ is a $\Z/h\Z$-suspension over  $(X_{\theta'},\sigma)$.
That is, $X_\theta \cong X_{\theta'}\times \Z/\sim$ with $(x,n+h)\sim(\sigma(x),n)$ and the action is induced by $\mathrm{id}\times (+1)$.
There is an explicit construction of $\theta'$ which in fact, equals $\theta$ if $h=1$.  Clearly, the maximal equicontinuous factor of the pure base system is $(\Z_\ell,+1)$.

Let   $\theta$ have as pure base  the substitution  $\theta'$ defined on the alphabet $\mathcal A'$. We say that $\theta$ has a {\em coincidence} if for some $k\in \N$ and some $i_1, \ldots ,i_k\in\{0,\ell-1\}$ we have $|\theta'_{i_1} \ldots \theta'_{i_k}(\mathcal A')|=1$. The importance of this notion lies in the theorem of Dekking stating that 
the substitution shift $( X_\theta,  \sigma)$ is almost automorphic if and only if $\theta$ has a  coincidence \cite{Dekking1977}.

 We shall see below that the question of whether $(X_\theta,\sigma)$ is tame or not is governed by the cardinality of the set of orbits of singular points in the maximal equicontinuous factor of $(X_\theta,\sigma)$. 
 Since $X_\theta \cong X_{\theta'}\times \Z/\sim$, the orbits of singular points in the maximal equicontinuous factor of $(X_\theta,\sigma)$ are in one-to-one correspondence with the orbits of singular points in the maximal equicontinuous factor of the pure base system $(X_{\theta'},\sigma)$.
 We may therefore just determine the cardinality of the latter.
 There is an effective procedure which achieves this \cite{CovenQuasYassawi2016}. We recapitulate a slightly modified version here in the only case which concerns us, 
which is when  $\theta$ has a coincidence and, by going over to the pure base of the substitution if needed,  its height is $1$. 

Consider the graph $\mathcal G_\theta$ whose
vertices are the sets  \[\{\mathcal A\}\cup 
\{A:=\theta_{w_1}\cdots \theta_{w_k}(\mathcal A)\colon k\ge 1,\;
w_1,\ldots,w_k\in \{0,1,\ldots,\ell-1\}, \mbox{ and } |A|>1\}\] 
and whose edges are defined as follows: If $A$ and  $B$ are vertices in $\mathcal G_\theta$, then there is 
an oriented edge from $B$ to $A$, labelled  $i$,  if and only if  $\theta_i(A)=B$. An infinite path in $\mathcal G_\theta$ defines a point $\dots z_2z_1\in\Z_\ell$, where $z_i$ is the label of the $i$-th edge in the path.

We define
$\hat \Sigma_\theta $ to be the set of all sequences $(z_i)\in\Z_\ell$ obtained as above from infinite paths in $\mathcal G_\theta$.
The $\Z$-orbit of $\hat\Sigma_\theta$ under  +1  equals $ \{ z\in \Z_\ell: |\pi^{-1}(z)|>1\}$, which is the set of singular points in $\Z_\ell$. 
This set is a proper subset of $\Z_\ell$, as $(X_\theta,\sigma)$ is almost automorphic.

Recall that a \emph{cycle} on an oriented graph is a finite path which is closed and minimal in the sense that it is not the concatenation of smaller closed paths.  

\begin{lem}\label{lem:finite-or-infinite} 
Let $\theta$ be a primitive aperiodic substitution of constant length with pure base $\theta'$. Its maximal equicontinuous factor contains  either finitely many, or uncountably many orbits of singular points. 
The latter is the case if and only if $\mathcal G_{\theta'}$ contains two distinct cycles which share a common vertex.
\end{lem}
\begin{proof} 
Observe that besides the paths corresponding to the orbit through $0$, two infinite paths in $\mathcal G_{\theta'}$ belong to the same $\Z$-orbit of $\Z_\ell$ if they differ only on a finite initial segment, that is,
if they are tail equivalent. 
Therefore, the statement comes down to showing that 
there are finitely many, or uncountably many distinct infinite paths up to tail equivalence in 
$\mathcal G_{\theta'}$. 

Clearly, $\mathcal G_{\theta'}$ must contain cycles as it contains infinite paths. If we have a vertex in two different cycles then, starting from that vertex, we can follow through the two cycles in any order we wish and therefore the number of paths in  $\mathcal G_{\theta'}$ is uncountable.

Now assume that there is no vertex in two different cycles.
An infinite path must visit some vertex infinitely often. 
As this vertex is not part of more than one cycle, the path must eventually follow the same cycle. 
As there are only finitely many vertices and edges, there can only be finitely many cycles. 
Hence, up to tail equivalence, there are only finitely many infinite paths in  $\mathcal G_{\theta'}$.  
\end{proof}

Let us anticipate the following important consequence of Lemma \ref{lem:finite-or-infinite}
combined with Theorem~\ref{Toeplitz finite rank non-tame} and the discussion  in Section~\ref{sec:toeplitz-constant}.

\begin{thm}\label{thm:restricted}
\label{thm: finite discontinuities iff tame}
Let $\theta$ be a primitive aperiodic substitution of constant length with pure base $\theta'$.
Then  $(X_\theta,\sigma)$ is tame if and only if it has a coincidence
and  $\mathcal G_{\theta'}$ does not contain two distinct cycles which share a common a vertex.
 \end{thm}

\begin{myexp}\label{ex:examplezero} Let $\theta$ be the substitution 
$$\begin{array}{c c l}
a & \mapsto & aaca\\
b & \mapsto & abba \\
c& \mapsto & aaba,
\end{array}
$$
on the alphabet $\Aa=\{a,b,c\}$.  It is primitive, aperiodic, and  has trivial height. Its graph $\mathcal G_\theta$ is depicted on the left hand side in Figure \ref{wrong_graph}. 
Here,
$\mathcal G_\theta$ has two different cycles at $\{a,b\}$, so by Theorem \ref{thm:restricted},  $(X_\theta,\sigma)$ is non-tame.
\end{myexp} 
\begin{myexp}\label{ex:examplezerob} 
We modify slightly the above example and define $\theta$ as
$$\begin{array}{c c l}
a &\mapsto  & aaca\\
b & \mapsto & abba \\
c& \mapsto & acba,
\end{array}
$$
on the alphabet $\Aa=\{a,b,c\}$. It still is primitive, aperiodic, and has trivial height.
 The graph $\mathcal G_{\theta}$ is shown on the right in Figure \ref{wrong_graph}, and as it has only one cycle about any vertex, 
by Theorem \ref{thm:restricted}, $(X_\theta,\sigma)$ is tame.
\end{myexp}

\begin{center}
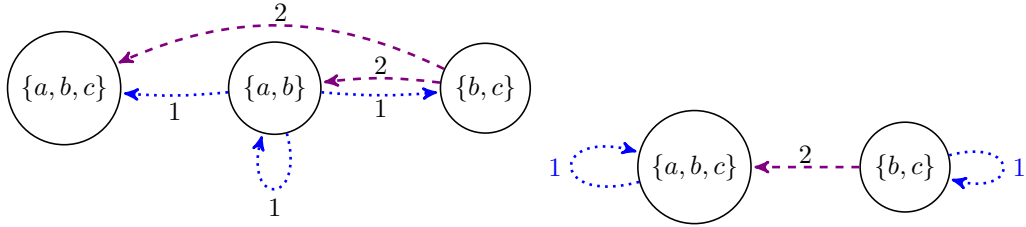
\begin{figure}
\begin{tikzpicture}
[->,>=stealth',shorten >=1pt,auto,node distance=2.8cm,
                    semithick] 
 \tikzstyle{every state}=[fill=white,draw=black,text=black]
 \node[state]         (A)  {$\{ a,b\}$};
   \node[state]         (O) [left of=A] {$\{ a,b,c\}$};

    \node[state]         (B) [right of=A]       {$\{b,c \}$};
  \path (B) edge[yshift=2pt, line width =1, bend right=7, color=violet, dashed] node[yshift=12pt, color=black]  { 2}  (A);
    \path (A) edge          [line width=1, bend right=5, color=blue, dotted]      node[yshift=-12pt, color = black]{1}(B);
   \path     (A) edge [loop below, line width =1, color=blue, dotted] node[color=black]{1} (A);
     \path (B) edge[yshift=2pt, line width =1, bend right=25, color=violet, dashed] node[yshift=12pt, color=black]  { 2}  (O);
        \path (A) edge          [line width=1, bend left=5, color=blue, dotted]      node[yshift=1pt, color = black]{1}(O);

                      \end{tikzpicture}
\begin{tikzpicture}[->,>=stealth',shorten >=1pt,auto,node distance=2.8cm,
                    semithick] 
 \tikzstyle{every state}=[fill=white,draw=black,text=black]
  \node[state]         (A)  {$\{ a,b,c \}$};
      \node[state]         (B) [right of=A]       {$\{b,c \}$};
   \path     (B) edge [loop right, line width =1, color=blue, dotted] node{1} (B);
     \path     (A) edge [loop left, line width =1, color=blue, dotted] node{1} (A);
       \path     (B) edge[yshift=2pt, line width =1, color=violet, dashed] node[yshift=12pt, color=black] {2} (A);
                      \end{tikzpicture}
\caption{\label{wrong_graph}
The graph $\mathcal G_\theta$ for the substitutions from Example~\ref{ex:examplezero} (left) and Example~\ref{ex:examplezerob} (right). We used blue dotted and violet dashed lines for better comparison with Figure \ref{thickness-not-rank-picture-1}.}
\end{figure}
\end{center}

\subsection{The Bratteli-Vershik representation of a  Toeplitz shift}
In this section, we briefly discuss the Bratteli-Vershik representation of Toeplitz shifts. 
As Bratteli-Vershik systems are well documented in the literature, we keep this to a minimum.
Interested readers may consult classical references on Bratteli-Vershik systems, such as \cite{HermanPutnamSkau}. 
In particular, since we will only be concerned with Bratteli-Vershik systems that are conjugate to Toeplitz shifts, we refer the reader to the work by Gjerde and Johansen \cite{GjerdeJohansen}; unless stated otherwise, we adopt the latter's notational conventions. 

\subsubsection{Bratteli-Vershik systems}
A {\em Bratteli diagram} is an infinite graph $B=(V,E)$ where the vertex
set $V=\bigsqcup_{n\geq 0}V_n$ and the edge set $E=\bigsqcup_{n\geq 0}E_n$
are equipped with a \emph{range map} $r\: E\to V$ and a \emph{source map} $s\: E\to V$ such that
\begin{enumerate}
 \item $V_0=\{v_0\}$ is a singleton,
 \item $V_n$ and $E_n$ are finite sets,
 \item $r(E_n)= V_{n+1}$, $s(E_n)= V_{n}$, 
 \item $r^{-1}(v)\neq\emptyset$ for all $v\neq v_0$. 
\end{enumerate}

The pair $(V_n,E_n)$ or just $V_n$ is called the \emph{$n$-th level} of the
diagram $B$.  A finite or infinite sequence of edges $(\gamma_n : \gamma_n\in E_n)$
such that $r(\gamma_{n})=s(\gamma_{n+1})$ is called a {\it finite} or {\it infinite
path}, respectively. 
The source map and the range map extend to paths in the obvious way.
For a Bratteli diagram $B$,
let $X_B$ be the set of infinite paths $(\gamma_n)_{n\geq 0}$ starting at the \emph{top
vertex} $v_0$. 
Given a path $(\gamma_n)_{n\geq 0}$ and $m\geq 0$, we call $(\gamma_n)_{n\geq m}$ a {\em tail of $(\gamma_n)$} and $(\gamma_n)_{n\leq m}$ a {\em head of  $(\gamma_n)$}.
Two paths $(\gamma_n)$ and $ (\gamma'_n)$ are called 
{\it tail equivalent} (or \emph{cofinal}) if they share a common tail.

We shall  constantly use the \textit{telescoping}
procedure.
Let $B$ be a Bratteli diagram, and $n_0 = 0 <n_1<n_2 < \ldots$ be a strictly increasing sequence of integers. The {\em telescoping of $B$ to $(n_k)$} is the Bratteli diagram $B'$, whose $k$-level vertex set is $V_k':= V_{n_k}$, and where the set of edges between $v\in V_k'$ and $w\in V_{k+1}'$ are in one-to-one correspondence with the set of paths in $B$ between $v\in V_{n_k}$ and $w\in V_{n_{k+1}}$.
There is then an obvious bijection 
between $X_B$ and $X_{B'}$.

A Bratteli diagram $B$ has {\em rank} $d$ if there is
a telescoping $B'$ of $B$ such that $B'$ has exactly $d$ vertices at each level. 
We say that $B$ is {\em simple} if for any level
$m$ there is $n>m$ such that for any two vertices  $v\in E_n$ and $w\in V_m$
there is a path with source $w$ and range $v$. This is equivalent to saying that by telescoping we can arrive at a Bratteli diagram $B'$ such that any two vertices in consecutive levels are connected by an edge.

Recall also that the path space $X_B$ comes with a totally disconnected compact metrizable topology, 
and if $B$ is simple and $|E_n|>1$ infinitely often, then $X_B$ is a Cantor set.
\begin{definition}
We say that $B=(V,E)$ has the \emph{equal path number property} if there is a sequence $(\ell_n)_{n\geq 0}$ such that for each $v\in V_{n+1}$ there are $\ell_n$ edges with range $v$. 
We call $(\ell_n)_{n\geq 0}$ the \emph{characteristic sequence} of $B$.
\end{definition}

\newcommand{\ord}{\om}

An {\em ordered} Bratteli diagram is a Bratteli diagram together with a total order on each set of edges which end at the same vertex. In other words, for each $r^{-1}(v)$, $v\in V$, the order naturally defines a bijection $\ord:r^{-1}(v)\to\{0,\ldots,|r^{-1}(v)|-1\}$. We refer to $\ord(e)$ also as the \emph{label} of $e$. Under certain circumstances,  given in detail in \cite{GjerdeJohansen}, this  order induces a {\em proper order} $\om$ 
on $X_B$. 
In a nutshell, the {\em successor} of an infinite  path $(\gamma_n)_{n\geq 0}\in X_B$ w.r.t.\ $\om$, when this exists,  is a tail-equivalent path whose order labelling corresponds to addition of $1$ to that of  $(\gamma_n)_{n\geq 0}$. The notion of a predecessor is defined analogously. The order being {\em proper} means that there is a unique path (the  {\em maximal} path) which has no successor in this order, and 
a unique path (the  {\em minimal} path) which has no predecessor. In this case one can define the {\em Vershik} map $\varphi_\omega : X_B\rightarrow X_B$, which sends a  nonmaximal path to its successor and which sends  the unique maximal path to the unique minimal path. It is a homeomorphism and thus defines a dynamical system $(X_B,\varphi_\omega)$ referred to as a {\em Bratteli-Vershik  system}.
If $(B,\omega)$ is a properly ordered  Bratteli diagram and $B'=(V',E')$ is the telescoping of $B$ to levels $(n_k)$, then  the order $\omega$ defines a natural proper order 
$\om'$ on $B'$, and $(X_B,\varphi_\om)$ is topologically conjugate to $(X_{B'},\varphi_{\om'})$. Properly ordered  simple Bratteli diagrams define minimal Bratteli-Vershik systems. Conversely, any Cantor minimal dynamical system $(X,T)$ is conjugate to a Bratteli-Vershik system where $B$ is simple \cite{HermanPutnamSkau}; the latter is called a Bratteli-Vershik representation of $(X,T)$. 
Not every such dynamical system has a Bratteli-Vershik representation with finite rank, but if this is the case, one says that 
$(X,T)$ has {\em finite topological rank}. More precisely, the topological rank of such a system is the smallest rank among its Bratteli-Vershik representations.

\subsubsection{Toeplitz Bratteli-Vershik systems}\label{sec: toeplitz bratelli-vershik systems}
We now focus on Bratteli diagrams which have the equal path number property. 
Here, if $v\in V_n$, then  $|r^{-1}(v)|=\ell_n$.
 Recall that a Bratteli-Vershik system is \emph{expansive} if and only if there is $k\in\N$ such that for distinct
 $x,y\in X_B$ the head of length $k$ of $\phi_\w^n(x)$ differs from that of $\phi_\w^n(y)$ for some $n\in\Z$.
 Downarowicz and Maass \cite{DownarowiczMaass2008} show that a simple properly ordered Bratteli-Vershik system 
 with finite topological rank is expansive if and only if its topological rank is strictly larger than $1$. A simple properly ordered Bratteli-Vershik system  with topological rank $1$ is an odometer. 
\begin{thm}[\cite{GjerdeJohansen}]\label{Gjerde-Johansen}
The family of expansive, simple, properly ordered Bratteli-Vershik systems 
with the equal path number property coincides with the family of  
Toeplitz shifts up to conjugacy. 
\end{thm}
In view of this result, we call an expansive simple properly ordered Bratteli-Vershik system
with the equal path number property a {\em Toeplitz Bratteli-Vershik system}. 
Moreover, we say that a Toeplitz shift has {\em finite Toeplitz rank} if it is conjugate to a Toeplitz Bratteli-Vershik system which has finite rank.
Note that having finite Toeplitz rank is stronger than having a Toeplitz shift with finite topological rank as we cannot rule out that a Toeplitz system with infinite Toeplitz rank has a Bratteli Vershik representation of finite rank but without the equal path number property.
\begin{lem} [\cite{GjerdeJohansen}]\label{lem-ord}
Let $(X_B,\varphi_\omega)$ be a Toeplitz Bratteli-Vershik system with characteristic sequence
 $(\ell_n)$. 
 The level-wise application of the \emph{edge order map} 
 \[\ord :(X_B,\varphi_\om) \to (\Z_{(\ell_n)},+1),\quad \ord((\gamma_n)) = (\ord(\gamma_n))\]
 is a factor map to the maximal equicontinuous factor.
\end{lem}

For Bratteli diagrams that have the equal path number property,  it is standard to  describe the ordering $\ord$ using a sequence of constant length morphisms.
To describe the ordering of the edge set $E_n$ between  $V_n$ and $V_{n+1}$, we use the morphism
 \begin{equation} \label{morphism-order}\theta^{(n)}:V_{n+1}\rightarrow V_n^{\ell_n}\end{equation}
which, when written as a concatenation of maps $\theta^{(n)}=\theta^{(n)}_0|\dots 
|\theta^{(n)}_{\ell_n -1}$, where $\theta^{(n)}_i:V_{n+1}\rightarrow V_n$
(similarly as in Section~\ref{sec: substitutions}), is given by
\[\theta^{(n)}_i(v) = s(e),\quad \mbox{with }\; e\in \ord^{-1}(i)\cap r^{-1}(v) ,\]
that is, the $i$-th morphism reads the source of the unique edge with range $v$ and label $i$.

If $(B,\omega)$ has the equal path number property, and $B'=(V',E')$ is the telescoping of $B$ to levels $(n_k)$, then $B'$  also has the equal path number property, and the corresponding morphism  
\begin{equation}\label{eq-morph}
\theta^{'(k)}:V_{k+1}' \rightarrow {V_k'}^{ \ell'_k}
\end{equation}   
is given by the composition
\[ \theta^{'(k)}= \theta^{(n_k)} \dots   \theta^{(n_{k+1}-1)}. \] 
We can again write $\theta^{'(k)}$ as a succession of maps {$\theta^{'(k)}_i:V_{k+1}'\rightarrow V_k'$ } which are each
a composition of maps $\theta\ho{n}_{i_n}:V_{n+1}\to V_n$, 
$n_k\leq n<n_{k+1}$ where the labels $i_n$ are those of the edges in $(B,\om)$ which constitute the edge $e\in E'_k$ with label $i$. More precisely, 
\[ \theta^{'(k)}_{i} 
= \theta^{(n_k)}_{i_{n_k}} \dots   \theta^{(n_{k+1}-1)}_{i_{n_{k+1}-1}} \quad \text{where} \quad
i = \sum_{n=n_k}^{n_{k+1}-1} i_n \prod_{m=n_k}^{n-1} \ell_{m}
 \]
(with the understanding that $\prod_{m=n_k}^{n_k-1} \ell_{m}=1$).
For instance, if we telescope level $n$ with level $n+1$, we get ${\theta'}\ho{n} = \theta\ho{n} \theta\ho{n+1}$ and
\[
{\theta'}_{i+j\ell_n}\ho{n} = \theta\ho{n}_i\theta\ho{n+1}_j \quad \text{for }
0\leq i < \ell_{n},\; 0\leq j < \ell_{n+1}.
\]

By telescoping if necessary, we can assume that  $\ell_n>1$ for each $n$. For, otherwise, $\ell_n=1$ for almost all $n$ and this implies that $X_B$ is 
not a Cantor space.

\subsubsection{Shift interpretation}\label{sec: shift interpretation}
There is a strong connection between expansive Bratteli-Vershik systems and shifts  \cite{DownarowiczMaass2008}.
Suppose $(X_B,\varphi_\om)$ is expansive and $k\in\N$ is such that the heads of length $k$ suffice to separate
orbits of distinct elements in $X_B$. 
Given a path $x\in X_B$, associate the bi-infinite sequence whose $n$-th entry consists of the head of length $k$ of $\varphi_{\om}^n(x)$. 
Let $(X_k,\sigma)$ be the shift whose space consists of all such sequences of length-$k$ heads. Then $(X_k,\sigma)$ and $(X_B,\phi_\w)$ are conjugate. 
By telescoping the diagram to the $k$-th level, we may assume that $k = 1$.
Given a Toeplitz Brattelli-Vershik system, we may (and will throughout this article) therefore assume that
\begin{equation}\label{def-projection-map} p_1:(X_B,\varphi_\om) \rightarrow (X_1,\sigma), \qquad (\gamma_\ell)_{\ell\geq 0} \mapsto {((\phi^n_\w(\gamma))_0)}_{n\in \Z}
\end{equation}
is a conjugacy.

If there is only one edge between the top vertex $v_0$ and each vertex in $V_1$, then the range map $r$ is a bijection between $E_0$ and $V_1$ so that $(X_B,\phi_\w)$ is conjugate to a shift over the alphabet $V_1$.
We denote the respective conjugacy also by $r$.
Observe that this situation can always be enforced by insertion of an extra level.  
Namely, we introduce an intermediate vertex set $\mathcal V$ between $V_0$ and $V_1$, which is in one-to-one correspondence with $E_0$. We introduce one edge from $v_0$ to each vertex of $\mathcal V$ and then for each $e\in E_0=\mathcal V$ an edge with source $e$ and range $r(e)\in V_1$ with the same order label as $e$. 
Clearly, the resulting Bratteli-Vershik system is topologically conjugate to the old one.
For Toeplitz Bratteli-Vershik systems this means that $\ell_0=1$. 

The following lemma is elementary to verify; its proof follows from the built-in recognizability of
Bratteli-Vershik systems, by which we mean that for each $n$, the towers defined by the first $n$ levels of $B$ form a partition of $p_1(X_B)$ where $p_1$ is defined in \eqref{def-projection-map} and which we assume, without loss of generality, to be a conjugacy. Given a sequence $(x_n)$ and $n<m$, we denote by $x_{[n,m[}$ the finite word $x_n x_{n+1}\ldots x_{m-1}$. 
\begin{lem} \label{lem-b} 
Let  $(X_B,\varphi_\om)$  be a  Toeplitz Bratteli-Vershik
system with characteristic sequence $(\ell_n)$ where $\ell_0=1$.
Set $\ell\ho{n} = \prod_{k=0}^{n-1}\ell_k$. 
For all $\gamma\in X_B$ and $n\in \N$, we have
$$ r\circ p_1(\gamma)_{[-z^{(n)},\ell^{(n)}-z^{(n)}[} = \theta^{(1)}\cdots \theta^{(n)}\circ r(\gamma_{n}),$$
where $z\ho{n} = \sum_{k=0}^{n-1} \ord(\gamma_k) \ell\ho{k}$.
In particular, 
\begin{equation}\label{eq-x}
 \{x_0  :  x\in  r\circ p_1\left(\w^{-1}(z)\right)\} \subset \bigcap_n   \theta^{(1)}_{z_1} \cdots \theta^{(n)}_{z_n} (V_{n+1}). 
 \end{equation}
\end{lem}

\subsubsection{Toeplitz Bratteli Vershik diagrams for constant length substitutions}\label{sec:toeplitz-constant}
 A {\em stationary}  Toeplitz Bratteli-Vershik system is one where for all $n\geq1$, $V_n=V_1$, $E_n=E_1$, and the order structure on $E_n$ is the same as that on $E_1$. 
In this case the morphisms $\theta^{(n)}$ of (\ref{eq-morph}) all agree and hence  
define a single morphism $\theta:V_1\to V_1^{\ell_1}$. If the range map is a bijection between $E_0$ and $V_1$ (there is a single edge between $v_0$ and each of the vertices of $V_1$) we can identify the space $X_1$ with the substitution shift space $X_\theta$ of $\theta$. In other words, a stationary  Toeplitz Bratteli-Vershik system for which $E_0\cong V_1$ defines a primitive substitution of constant length. 

The converse, associating a Toeplitz Bratteli Vershik system to a primitive substitution $\theta$ of constant length $\ell$ over an alphabet $\Aa$ is subtle. The natural approach \cite{Livshits-Vershik},
which consists of defining 
a stationary  Bratteli-Vershik system by setting $V_n=\Aa$ and defining the edges and their order with $\theta^{(n)}=\theta$ as in (\ref{eq-morph}), works well if all substitution words start with the same letter and end with the same letter, that is,
$\theta_0(a)$ and $\theta_{\ell-1}(a)$ are independent of $a$. 
Indeed, if that is the case, then the order on the Bratteli diagram is proper.
However, for general $\theta$, the respectively defined order may fail to be proper, and the arguments in \cite{Livshits-Vershik} only give a measurable Bratteli-Vershik representation.
While there are classical methods to rewrite the substitution to obtain a stationary Brattelli-Vershik representation for $(X_\theta,\sigma)$ \cite{DurandHostSkau1999, Forrest1997}, those procedures do not necessarily give a Toeplitz Bratteli-Vershik representation.
Instead one needs to follow the  approach of \cite{GjerdeJohansen}  to obtain a properly ordered Toeplitz Bratteli-Vershik system such that $(X_1,\sigma)$ equals $(X_\theta, \sigma)$.

%%%%%%%%%%%%%%%%%%%%%%%%%%%%%%%%%
\subsection{The extended Bratteli diagram}\label{extended-BV}
%%%%%%%%%%%%%%%%%%%%%%%%%%%%%%%%%%
In this section we introduce the notion of the extended Bratteli diagram, and its essential thickness, concepts which are fundamental for the proofs of Theorem~\ref{thm:0} and Theorem~\ref{thm: A}.
The extended Bratteli diagram can be seen as a generalisation of the graph $\mathcal G_\theta$,  introduced in Section \ref{odometer-toeplitz-substitution} for constant length substitution shifts.  

\begin{definition}\label{extended-diagram-definition}
Let  $(B, \om)$ be  an ordered Bratteli diagram with the equal path number property and characteristic sequence $(\ell_n)_{n \geq 0}$.
The extended Bratteli diagram is an infinite graph  which satisfies the properties (1)--(3) of a Bratteli diagram, but not necessarily (4). The 
 {\em extended Bratteli diagram} $\tilde B=(\tilde V_n,\tilde E_n)$ associated to $B$ has 
the following vertices and edges: 
\begin{enumerate}
\item The level $n$ vertex set $ \tilde V_n$ is the set of all nonempty subsets of $V_n$.
\item For $0\leq i < \ell_n$,  $\tilde E_n$ contains an edge 
labelled $i$  with source $A\in \tilde V_{n}$ and range $B\in \tilde V_{n+1}$ if  and only if $\theta^{(n)}_{i}(B) = A$. 
\end{enumerate}
\end{definition}
Identifying singleton sets with the element they contain, we see that $\tilde B$ contains 
$B$ as a sub-diagram. The labelling of the edges defines an order on $\tilde B$ which extends the order on $B$. We also consider the space $X_{\tilde B}$ of infinite paths over $\tilde B$ starting at the top vertex $v_0$.
Clearly, $X_{\tilde B}$ contains $X_B$. The edge order map $\ord$ from Lemma~\ref{lem-ord} extends to a map from $X_{\tilde B}$ to the maximal equicontinuous factor $\Z_{\ell_n}$ 
which we denote by the same letter $\ord$.
Due to the lack of property (4), the vertices in the extended diagram need not to have any outgoing edges. Infinite paths ignore such vertices and we call vertices {\em extendable} if they are traversed by a path in $X_{\tilde B}$.

To a path $\gamma$ in $X_{\tilde B}$ we associate the sequence of maps 
\begin{equation}\label{eq-theta}
\th{n}{\gamma}:=\th{n}{\ord(\gamma)_n  }: \tilde V_{n+1}\to \tilde V_n
\end{equation}
and the sequence of subsets $A_n:=s(\gamma\ho{n}) \subset V_n$. Then 
$\th{n}{\ord(\gamma)_n  }(A_{n+1})=A_{n}$. 
Recall that we can arrange for $\ell_0=1$ in the characteristic sequence of the original  
Toeplitz Bratteli system. This implies that in the extended Bratteli diagram the top vertex is linked to any vertex of $\tilde V_1$ by exactly one edge.  

We remark that in the case where the ordered diagram $B$ is {\em stationary}, that is, $E_n=E_1$ and $\theta^{(n)}=\theta^{(1)}$ for $n\geq 1$, then the graph $\mathcal G_{\theta^{(1)}}$ defined in Section \ref{odometer-toeplitz-substitution} 
is an abbreviated form of the extended Bratteli diagram $(\tilde B ,\tilde \omega)$. For, the extended Bratteli diagram will also be stationary and so can be described by the edge and order structure of its first level. The only other difference is that in $\mathcal G_{\theta^{(1)}}$, we chose to exclude vertices indexing one-element sets, as $\mc G_{\theta^{(1)}}$ is only to  identify the singular fibres. 

Note that since $|\th{n}{i}(A)|\leq |A|$ for each $n$ and $i$, a path in $X_{\tilde B}$ must pass through vertices of non-decreasing cardinality. This motivates the following definition.

\begin{definition}
If
 a path of $X_{\tilde B}$ 
eventually goes through vertices $A_n\in \tilde V_n$ with $|A_n|=k<\infty$ for all $n$ large, we will say that  the path has {\em thickness $k$}. Otherwise we say 
that the path has \emph{infinite thickness}.
\end{definition}
If the diagram has finite rank, then there are no paths with infinite thickness. 
We denote by $X_{\tilde B}^{k}$ the infinite paths of thickness $k\in \N\cup\{\infty\}$. 
Note that the subdiagram  $X_{\tilde B}^{1}$ corresponds to the original path space $X_B$.

%%%%%%%%%%%%%%%%%%%%%%%%%%%%%%%%%%%%%%%%%

The following lemma tells us  that $z\in  \Z_{(\ell_n)}$ is singular if and only if it is the image of a path of thickness $k>1$.
Let $\thick(\gamma)$ denote the thickness of the path $\gamma$.
\begin{lem}\label{lem-x} 
Let $(X_B,\varphi_\om)$ be a Toeplitz Bratteli-Vershik system  with extended path space $X_{\tilde B}$. Let $z\in\Z_{(\ell_n)}$.
Then
$$|\{\gamma\in X_{B} : \ord(\gamma)=z\}| = \sup\{\thick(\tilde\gamma) : \tilde\gamma\in X_{\tilde B}, \ord(\tilde\gamma) = z\}.$$
In particular, the set of singular points in $\Z_{(\ell_n)}$ coincides with the union 
$\bigcup_{j\geq 2}\ord(X^j_{\tilde B})$, and 
the rank of the Toeplitz Bratteli-Vershik system is an upper bound for the maximal number of elements in a fibre of the factor map to the maximal equicontinuous factor.
\end{lem}
\begin{proof} 
Recall the definition of the maps $\th{n}{\tilde\gamma}$ and subsets $A_n:=s(\gamma\ho{n}) \subset V_n$ associated to a path  
$\tilde\gamma\in X_{\tilde B}$ in \eqref{eq-theta}. 
If $\thick(\tilde\gamma)\geq k$, then there exists $n_0$ such that  $|A_{n}|\geq k$ for  $n\geq n_0$. 
 Since $  \th{n}{\tilde\gamma}     (A_{n+1})=A_{n}$, there are at least $k$ paths $\gamma$ in the original path space $X_B$ such that $\ord(\tilde\gamma) = \ord(\gamma)$.
 This shows the inequality ``$\geq$".

Now suppose that $|\{\gamma\in {X_{B}}: \ord(\gamma)=z\}| \geq k$, so that there are at least $k$ distinct paths $\gamma {\in X_{ B}}$ with the same edge labels. We need to make sure that there is at least one $n$ such that they go through $k$ different vertices at level $n$. Note that if two paths with equal edge labels agree on a vertex at level $n$ then their head agrees up to level $n$. Thus $k$ distinct paths must at some level go through $k$ distinct vertices. This implies that there is an $A_{n}$ with $|A_{n}|\geq k$ which is a vertex of a path in $X_{\tilde B}$ which has edge labels $z$. 
Thus $\sup\{\thick(\tilde\gamma):\tilde\gamma\in X_{\tilde B}, \ord(\tilde\gamma) = z\}\geq k$.
\end{proof}

\subsection{Thick Toeplitz shifts}
A pair of {\em parallel edges} in $\tilde E_n$ is a pair of edges 
$(e_{1}, e_{2})\in \tilde E_n\times \tilde E_n $ with the same source and range but  distinct labels according to the order.
A {\em double path} in $X_{\tilde B}$ is a pair of paths consisting of parallel edges at each level $n > 0$. 
We write
$\db{\gamma}=( \db{\gamma}_n)=(\gamma_{n,1}, \gamma_{n,2})$ to denote a double path. 
%%%%%%%%%%%%%%%%%%%%%%%%
%%%%%%%%%%%%%%%%%%%%%%%
\begin{definition}\label{def:thick}
The largest $k$ such that $X_{\tilde B}^{k}$ is uncountable is called the {\em essential thickness} of $(X_B,\varphi_\omega)$.
We say that $(X_B,\varphi_\omega)$ is {\em thick} if 
its essential thickness $k$ is strictly larger than $1$ and finite,  and if there is a double path of thickness $k$.
A Toeplitz shift is {\em thick} if it has a  thick Toeplitz Bratteli-Vershik representation.
\end{definition}
\renewcommand{\thmcontinues}[1]{continued}
\begin{myexp}[continues=ex:examplezero]  \label{ex:first} 
 To illustrate the above notions  we apply them to the first substitution in Example \ref{ex:examplezero}. 
 While this example is not sensitive to some of the subtleties that we will meet later (because of its stationarity) it can at least be described explicitly.

Recall that $\theta: \{a,b,c \}\rightarrow \{a,b,c \}^{4}$ is the substitution 
$$\begin{array}{c c l}
a & \mapsto & aaca\\
b & \mapsto & abba \\
c & \mapsto  & aaba. \\
\end{array}
$$
Since all substitution words begin and end on $a$, the approach of \cite{Livshits-Vershik} to define the Toeplitz Bratteli Vershik system works here and it is not difficult to derive the extended system as well.
The extended Bratteli diagram is stationary and we have drawn one level
in Figure \ref{thickness-not-rank-picture-1}. We follow the convention of reading levels from top to bottom.
  Note that the vertices $\{a,b,c\}$ and $\{ a,c\}$ have no outgoing edges, so no infinite path will go through them and, in particular, there are no paths of thickness $3$. There are uncountably many paths of thickness $2$, namely those which keep going through vertices $\{a,b\}$ or $\{b,c\}$. 
Hence all singular fibres consist of  two elements and the essential thickness is $2$.

If we telescope the extended Bratteli diagram to even levels we will find that there are two edges between two consecutive vertices  $\{a,b\}$. These two edges form a pair of parallel edges and consequently the telescoped diagram admits a double path of thickness $2$.
In particular, the Toeplitz Bratteli-Vershik system is thick.

Notice the connections between Figure  \ref{thickness-not-rank-picture-1} and the graph $\mathcal G_\theta$ in Figure \ref{wrong_graph}.
The infinite paths on $\mathcal G_\theta$ correspond to paths in the stationary extended Bratteli diagram which start at the top vertex $v_0$ (the first level consists of one edge between $v_0$ and each of the seven vertices of $\tilde V_1$) and go downwards without ever passing through a vertex which is a singleton nor through a vertex which does not have an outgoing edge. The fact that the telescoped extended Bratteli diagram admits a double path of thickness $2$ going through the vertices $\{a,b\}$ is equivalent to the fact that the vertex $\{a,b\}$ of $\mathcal G_\theta$ belongs to two distinct cycles.
\end{myexp}
\begin{myexp} 
[continues=ex:examplezerob] 
It is not difficult to derive an extended Bratteli diagram for the substitution of Example~\ref{ex:examplezerob} as well.  What one will find is that there is one edge between two consecutive vertices $\{a,b,c\}$ and one edge between two consecutive vertices $\{b,c\}$. It follows that the diagram has thickness $3$.
However, there is only one path which goes infinitely often through $\{a,b,c\}$ and only countably many which go infinitely often through $\{b,c\}$. It follows that the essential thickness of the diagram is $1$. The system is therefore not thick.
\end{myexp}

\definecolor{red(munsell)}{rgb}{0.95, 0.0, 0.24}
\definecolor{mediumelectricblue}{rgb}{0.01, 0.31, 0.59}
 \definecolor{lavender(floral)}{rgb}{0.71, 0.49, 0.86}
\definecolor{lavenderblue}{rgb}{0.8, 0.8, 1.0}
\definecolor{darkgreen}{rgb}{0.0, 0.2, 0.13}
\definecolor{cadmiumgreen}{rgb}{0.0, 0.42, 0.24}
  \begin{center}
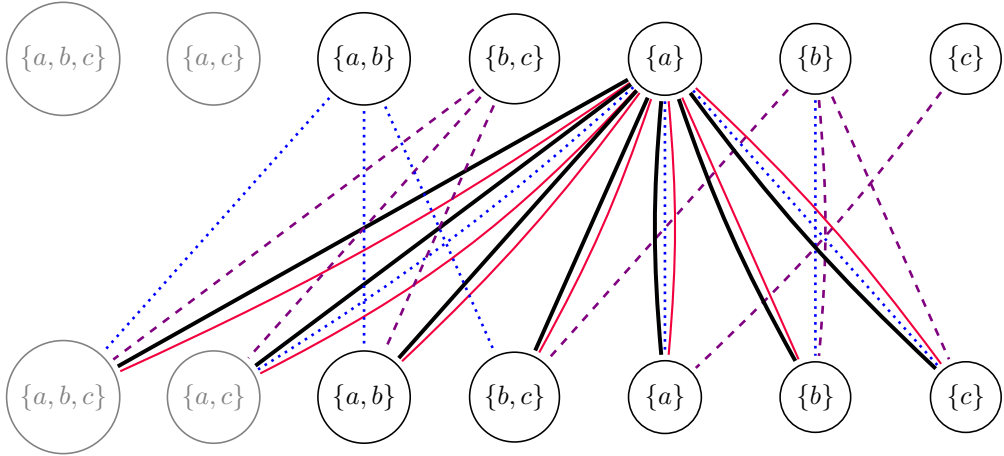
\begin{figure}
\begin{tikzpicture}
[-,>=stealth',shorten >=2pt, shorten <=2pt,auto,node distance=2.0cm,semithick] 
 \tikzstyle{every state}=[fill=white,draw=black,text=black]
  \node[state, draw=gray]         	(A)  		[text=gray]	{$\{ a,b,c\}$};
  \node[state, draw=gray]         	(a)  [right of=A, text=gray]	{$\{ a,c\}$};
  \node[state]         	(B)  [right of=a]	{$\{a,b\}$};
  \node[state] 		(C)  [right of=B]       {$\{b,c \}$};
  \node[state] 		(D)  [right of=C]       {$\{a\}$};
  \node[state] 		(E)  [right of=D]       {$\{b\}$};
  \node[state] 		(F)  [right of=E]       {$\{c\}$};

  \node[state, draw=gray]         	(X)   at (0,-4.5)	[text=gray]	{$\{ a,b,c\}$};
  \node[state, draw=gray] 		(x)   [right of=X, text=gray]      {$\{a,c \}$};  %at (0,-4) 
  \node[state]        	(B1)  [right of=x]     	{$\{a,b\}$};
  \node[state] 		(C1)  [right of=B1]     {$\{b,c \}$};
  \node[state] 		(D1)  [right of=C1]     {$\{a\}$};
  \node[state] 		(E1)  [right of=D1]     {$\{b\}$};
 \node[state] 		(F1)  [right of=E1]     {$\{c\}$};

      \path

  (D) edge [line width=1.5, ]        (X)     % node[pos=0.3]  
         (D) edge [line width=0.8, bend left=5,  color=red(munsell)]         (X)     % node[pos=0.3]  
         (B) edge [line width =1, color=blue, dotted] (X)
           	(C) edge [line width =1, dashed, color=violet] (X)

	(D) edge [line width=1,  bend left=5, color=blue, dotted]       (x) 
	(D) edge [line width=0.8,  bend left=10, color=red(munsell)]        (x) 
	  (D) edge [line width=1.5,  ]         (x) 
	  (C) edge [line width =1,  color=violet, dashed, ] (x)

 (B) edge  [line width =1, color=blue, dotted]  (B1)
         (C) edge  [line width =1, dashed, color=violet] (B1)
        (D) edge [line width=1.5,color=black]          (B1)     %  node[pos=0.3] 
          (D) edge [line width=0.8,  bend left=5,  color=red(munsell)]           (B1)  
        
	(D)  edge [line width=1.5, color=black]        (C1) 
	(D)  edge [line width=0.8, bend left=5,  color=red(munsell)]      (C1) 
	(B) edge  [line width =1, color=blue, dotted]  (C1)
        (E) edge [line width =1, dashed, color=violet] (C1)

	(D)  edge [line width=1.5, bend right=5]               (D1) 
	(D)  edge [line width=0.8,  color=red(munsell), bend left=5]                (D1) 
	  (D) edge [line width=1, color=blue, dotted]  (D1)
(F)  edge [line width=1, bend left=5, dashed, color=violet]             (D1)

	(D) edge [line width=0.8, color=red(munsell)]   	(E1)   
	(D) edge [line width=1.5, bend right=5]   	(E1)   
	  (E) edge [line width =1, color=blue, dotted] (E1)     
	     (E) edge [line width =1, dashed, color=violet, bend left=5] (E1)

        (D) edge [line width=1.5, bend right=5]      (F1)       
        (D) edge [line width=1, color=blue, dotted, ]      (F1)    
        (D) edge [line width=0.8,  color = red(munsell), bend left=5]      (F1)    
	
       (E) edge [line width =1, dashed, color=violet] (F1);

\end{tikzpicture}
\caption{\label{thickness-not-rank-picture-1}One level of the stationary extended Bratteli diagram of  Example \ref{ex:first}. The order is indicated through colour:
black, blue dotted, violet dashed and red edges correspond to order label 0,1,2 and 3 respectively. 
The grey vertices are not extendable. Red edges are finer than black edges if viewed without colour.}
\end{figure}
\end{center}

\begin{myexp}\label{exp:oxtoby}
Oxtoby \cite{Oxtoby:1952} described a family of minimal binary Toeplitz shifts that are not uniquely ergodic and hence cannot be tame. 
We describe the Bratteli-Vershik representations for the one-sided versions of this family to maximise the similarity to his original description. Given a sequence $(\ell_n)$ of natural numbers, define the substitutions
 \begin{align*}
  a&\stackrel{\theta^{(n)}}{\mapsto} ab^{\ell_n -1}\\
  b&\stackrel{\theta^{(n)}}{\mapsto} aa^{\ell_n-1},
  \end{align*}
  and define an ordered Bratteli diagram with the sequence $\{\theta^{(n)} \}$ as in \eqref{morphism-order}; see Figure  \ref{figure:oxtoby} for two examples, the one on the left  with $\ell_1=\ell_2=2$, the second with $\ell_1=3$, $\ell_2=5$.

  Note that these ordered Bratteli diagrams each have two maximal paths and one minimal path; this means that we cannot define a Vershik map which is a homeomorphism. Nevertheless we can still define a continuous Vershik map by sending the two maximal paths to the unique minimal path. This one-sided Bratteli-Vershik system is conjugate to the one-sided shift defined by Oxtoby.
  Oxtoby showed that if  $(\ell_n)$ grows fast enough, in particular if $\sum \frac{\ell_{k-1}}{\ell_k}<1$, then the resulting system is not uniquely ergodic, and thus it cannot be tame.
  In the right hand side of Figure \ref{oxtoby-BV}, one clearly sees the beginning of a double path of thickness two. This should be contrasted with the stationary figure on the left, which is a one-sided representation of the {\em period-doubling} substitution shift, and which has thickness one, and so is tame, as we will see below.
\end{myexp}

 \begin{figure}\label{figure:oxtoby}
 \hspace{-58pt}
\begin{minipage}{0.5\textwidth}
   \begin{tikzpicture}
[-,>=stealth',shorten >=2pt, shorten <=2pt,auto,node distance=2.0cm,semithick] 
 \tikzstyle{every state}=[fill=white,draw=black,text=black]
   \node[state]         	(X)  			{*};
  \node[state]         	(A)  	 [below left of=X]		{$\{a\}$};
  \node[state]         	(B)  [below right of=X]	{$\{b\}$};
   \node[state]         	(C)  [ right of=B]	{$\{a,b\}$};
  \node[state]        	(A1)  [below of=A]     	{$\{a\}$};
  \node[state] 		(B1)  [below of=B]     {$\{b\}$};
   \node[state] 		(C1)  [right of=B1]     {$\{a,b\}$};
  \node[state] 		(A2)  [below of=A1]     {$\{a\}$};
  \node[state] 		(B2)  [below of=B1]     {$\{b\}$};
   \node[state] 		(C2)  [right of=B2]     {$\{a,b\}$};
 \path
  (X) edge [line width=1, ]        (A)
   (X) edge [line width=1, ]        (B)
     (X) edge [line width=1, ]        (C)
    (A) edge [line width=1.5, ]         (A1)     
         (A1) edge [line width=1.1, color=blue, dotted ]          (B)     
          (B1) edge [line width =1.5, bend left = 4]   (A)
           	(B1) edge [line width =1.1, bend right =4,  color=blue, dotted]   (A)
	(A1) edge [line width=1.5, ]        (A2)     
         (A2) edge [line width=1.1, color=blue, dotted ]  node[yshift=-45pt, xshift=8pt, color=black]  { $\vdots$}  
           node[yshift=-57pt, xshift=8pt, color=black]  { $\vdots$}  
 (B1)

           (B2) edge [line width =1.5, bend left = 4]   (A1)
           	(B2) edge [line width =1.1, bend right =4, color=blue, dotted] (A1)
	 (C) edge [line width=1.1, color=blue, dotted ] (C1)
	  (C1) edge [line width=1.1, color=blue, dotted ]  (C2)    
	   (C1) edge [line width=1.5, ]  (A)  
	    (C2) edge [line width=1.5, ] (A1);     
           \end{tikzpicture}
          \end{minipage}
  \begin{minipage}{0.3\textwidth}
  \begin{tikzpicture}
[-,>=stealth',shorten >=2pt, shorten <=2pt,auto,node distance=2.0cm,semithick] 
 \tikzstyle{every state}=[fill=white,draw=black,text=black]
   \node[state]         	(X)  	 	{*};
  \node[state]         	(A)  	 [below left of=X]	{$\{a\}$};
  \node[state]         	(B)  [below right of=X]	{$\{b\}$};
   \node[state]         	(C)  [right of=B]	{$\{a,b\}$};

   \node[state]        	(A1)  [below of=A]     	{$\{a\}$};
  \node[state] 		(B1)  [below of=B]     {$\{b\}$};
   \node[state]         	(C1)  [right of=B1]	{$\{a,b\}$};

  \node[state] 		(A2)  [below of=A1]     {$\{a\}$};
  \node[state] 		(B2)  [below of=B1]     {$\{b\}$};
   \node[state]         	(C2)  [right of=B2]	{$\{a,b\}$};

 \path

  (X) edge [line width=1, ]        (A)
   (X) edge [line width=1, ]        (B)
    (X) edge [line width=1, ]    	    (C)
    (A1) edge [line width=1.5]      (A)     
         (A1) edge [line width=0.8, color=blue, dotted]        (B)     
         (A1) edge [line width=0.8, bend right=5, dashed, color=violet]      (B)     
                 (B1) edge [line width =1.5,  bend left=5] (A)
           	(B1) edge [line width =1, color=blue, dotted]  (A)
	(B1) edge [line width =1, bend right =5, dashed, color=violet]   (A)
	(C1) edge [line width =1.5]  (A)
	(C1) edge [line width =1, bend left=5, color=blue, dotted]   (C)
	(C1) edge [line width =1, bend right=5, color=violet, dashed]  (C)
		(A1) edge [line width=1.5, ]         (A2)

         (A2) edge [line width=0.8, bend left=5, color=blue, dotted]         (B1)     
         (A2) edge [line width=0.8, dashed, color=violet  ]      (B1)     
         (A2) edge [line width=0.8, bend right=5, densely dotted, color=darkgreen]      (B1)     
         
         (A2) edge [line width=0.8, bend right=10, color=red]      node[yshift=-43pt, xshift=8pt, color=black]  { $\vdots$}  
           node[yshift=-55pt, xshift=8pt, color=black]  { $\vdots$}     
           (B1)   
           (B2) edge [line width =1.5,bend left=10]  (A1)
             (B2) edge [line width =1, bend left=5, color=blue, dotted ]   (A1)
               (B2) edge [line width =1, dashed, color=violet ]   (A1)
                 (B2) edge [line width =1,bend right=5, densely dotted, color=darkgreen]  (A1)
                   (B2) edge [line width =1,bend right=10, color=red]   (A1)
                                       (C2) edge [line width =1.5]  (A1)
                     (C2) edge [line width =1,  bend left=5, color=blue, dotted] (C1)
           	(C2) edge [line width =1,color=violet, dashed ]  (C1)
	(C2) edge [line width =0.8, bend right =5, densely dotted, color=darkgreen]   (C1)
	(C2) edge [line width =0.8, bend right =10, color=red]  (C1);
           \end{tikzpicture}
          
           \end{minipage}
           \caption{\label{oxtoby-BV} 
                On the right,   we see the first levels of the extended Bratteli diagram of  the one-sided shifts for Example \ref{exp:oxtoby} with $\ell_1=3$ and $\ell_2=5$. The order is indicated through colour:
black, blue dotted, violet dashed, green densely dotted, and red edges correspond to order label $0,1,2,3$ and $4$ respectively. 
Red edges are finer than black edges if viewed without colour. As more levels are added, there are increasingly many edges between vertices labelled $\{a,b\}$ in consecutive levels.
This is to be contrasted with the  one-sided {\em period-doubling} substitution shift (on the left), where $\ell_n=2$ for all $n$, and which is tame (again, black  and blue dotted edges correspond to order label $0$ and $1$, respectively). It has thickness one.
}

           \end{figure}
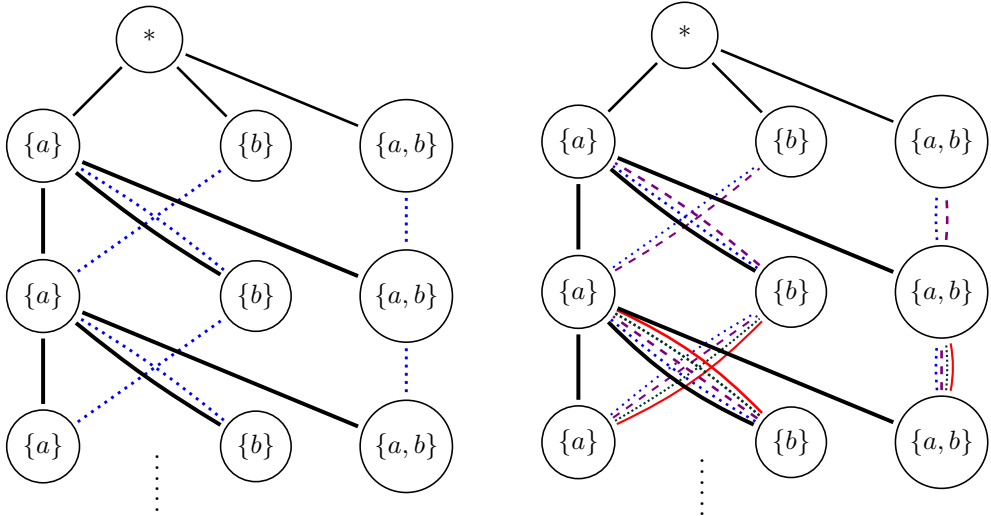

Recall that $x\in X$ is a {\em condensation point } if every neighbourhood of $x$ is uncountable.  Let $\mathcal C\subset X$ be the set of its condensation points. The Cantor-Bendixson theorem gives that for second countable spaces, $X\backslash \mathcal C$ is countable.

\begin{lem}\label{lem:thick-double-path} 
Let $(X_B,\varphi_\w)$ be a Toeplitz Bratteli-Vershik system.
\begin{enumerate}
\item If all paths in $X_{\tilde B}$ have finite thickness, the maximal equicontinuous factor contains uncountably many singular points if and only if there is $j>1$ such that $X_{\tilde B}^{j}$ is uncountable.
\item If $X_{\tilde B}^{j}$ contains a double path, then it is uncountable.
\item If $X_{\tilde B}$ has finite rank and $X_{\tilde B}^{j}$ is uncountable, then $X_{\tilde B}^j$ contains a double path.
\end{enumerate}
\end{lem}
\begin{proof}
For $z\in \Z_{(\ell_n)}$ and $n\in \N$, let $B^z_n\ssq V_n$ be of maximal cardinality among those elements of $\tilde V_n$ which are traversed by some 
path $\tilde\gamma\in X_{\tilde B}$ with $\w(\tilde\gamma)=z$.
Observe that $B^z_n$ is uniquely determined because if $B,B'\in V_n$ are traversed by such  a $\tilde\gamma$, then $B\cup B'$  is extendable and   there is such a $\tilde\gamma$ going through $B\cup B'$.
This shows that the supremum in the formula of Lemma~\ref{lem-x} is attained at some path $\tilde \gamma$.
Since all paths in $X_{\tilde B}$ have finite thickness, Lemma~\ref{lem-x} implies that the restriction of $\w$ to $X_B$ is finite-to-one, and this implies that the restriction of $\w$ to $X^j_{\tilde B}$ is finite-to-one.
Hence if $X_{\tilde B}^{j}$ is uncountable, then its image under $\w$ must be uncountable.
The converse follows from Lemma~\ref{lem-x}, which tells us that the singular points of $\Z_{(\ell_n)}$ are given by the image of $X_{\tilde B}\setminus X_{\tilde B}^1=\bigcup_{j\geq 2} X_{\tilde B}^j$.

Suppose that $X_{\tilde B}^{j}$ contains a double path. 
Then there is $n_0$ such that for all $n\geq n_0$
there is $A_n\subset V_n$ containing $j$ elements such that between $A_n$ and $A_{n-1}$ there are at least $2$ edges in the extended Bratteli diagram.
The set of paths in $X_{\tilde B}^{j}$
obtained by choosing one of the two edges  at each level  is uncountable. 

Suppose now that $X_{\tilde B}^{j}$ is uncountable. 
Then the set of condensation points $X_{\tilde B}^{j}\cap \mathcal C   $ of $X_{\tilde B}^{j}$ is uncountable. 
Since $X_{\tilde B}^{j}\cap \mathcal C   $  has no isolated points, for any given path $\gamma\in  X_{\tilde B}^{j}\cap \mathcal C     $ and $n\geq 0$ there are infinitely many distinct paths in  $X_{\tilde B}^{j}\cap \mathcal C   $ which agree with $\gamma$  on its head of length $n$. 
Let $K$ be the rank of the Toeplitz-Bratteli-Vershik system and pick some $m_1\geq 1$. There is $m > m_1$ and $K+1$ paths of $X_{\tilde B}^{j}\cap \mathcal C   $ which agree with  $\gamma$'s head of length $m_1$ but pairwise disagree on the head of length $m$. 
Since infinitely often $|V_n|=K$, the pigeon hole principle requires that two of the distinct paths must meet a common vertex of level $m_2>m$. Telescoping the levels $m_1$ through $m_2$,  the part of these two paths between level $m_1$ and level $m_2$ defines a parallel edge. 
Iterating this procedure ($m_2$ playing the role of $m_1$ etc.) proves the statement.
\end{proof}

The following corollary can be seen as a generalisation of Lemma~\ref{lem:finite-or-infinite} to all Toeplitz shifts with finite rank.

\begin{cor}\label{thick-double-path} 
Let $(X_B,\varphi_\w)$ be a Toeplitz Bratteli-Vershik system with finite Toeplitz rank. The maximal equicontinuous factor contains uncountably many singular points if and only if $(X_B,\varphi_\w)$ is thick.
\end{cor}
\begin{proof}
Finite rank implies that all paths in $X_{\tilde B}$ are of finite thickness. 
Now, part (1) of Lemma~\ref{lem:thick-double-path}
gives that we have uncountably many singular points if and only if the essential thickness is strictly greater than $1$. 
Part (3) of the same lemma gives that if the essential thickness is $k$, then there is a double path of thickness $k$.
\end{proof}

Given a choice function $\varphi \in \{0,1\}^{\N_0}$ and a double path $\db{\gamma}$, let $\varphi(\db{\gamma})$ be the (single edge) path $ ({\gamma}_{n,\varphi(n)})$. Such a single edge path defines a sequence of maps $\th{n}{\varphi(\db{\gamma})}$, see (\ref{eq-theta}).
The second part of Lemma~\ref{lem:thick-double-path} implies that the essential thickness is an upper bound for the maximal thickness a double path can have.   
From the next results, we obtain more delicate information:
the size of the image $ | \th{m}{\varphi(\db{\gamma})}\cdots  \th{n}{\varphi(\db{\gamma})} (V_{n+1})|$ of sufficiently long finite paths is also bounded above by the essential thickness.

In the proof of the next statement, we denote the
set of all finite and infinite paths in $\tilde B$ which start at the top vertex 
 $v_0$ by $Y_{\tilde B}$.
Observe that $Y_{\tilde B}$ can naturally be seen as a compact  space, where finite paths are seen as infinite paths which eventually pass through a placeholder vertex.

\begin{lem}\label{all-choice-bounded}
Let $(X_B,\varphi_\w)$ be a Toeplitz Bratteli-Vershik system with essential thickness $k$. 
Consider $z^0,z^1\in \Z_{(\ell_n)} $ with
$z^0_n\neq z^1_n$ for each $n$.
Then for all but at most countably many $\varphi\in \{0,1\}^{\N_0}$ we have
\begin{align*}
 \forall m\in \N \,\, \exists n_0>m \,\, \forall n\geq n_0 \:
| \th{m}{z^{\varphi(m)}_m}\cdots  \th{n}{z^{\varphi(n)}_n} (V_{n+1})| \leq k.
\end{align*}
 \end{lem}
\begin{proof}
We only have to consider the case of finite $k$ as the statement is trivially  true otherwise.
Suppose there are uncountably many $\varphi$ such that
$$ \exists m\in \N\,\, \forall n_0>m\,\, \exists n\geq n_0\:
|\th{m}{z^{\varphi(m)}_m}\cdots  \th{n}{z^{\varphi(n)}_n} (V_{n+1})|> k.$$
Then for each such $\varphi$ there are arbitrarily large $n$ and paths $\gamma^{\varphi;n}\in Y_{\tilde B}$ of length $n$ with $\w(\gamma^{\varphi;n}_\ell)=z^{\varphi(\ell)}_\ell$ for $\ell=0,\ldots,n-1$ which traverse a subset of $V_m$ of size bigger than $k$.
Due to the compactness of $Y_{\tilde B}$ there must hence be an infinite
path $\gamma^\varphi$ (i.e., an element of $X_{\tilde B}$) with $\w(\gamma^{\varphi}_\ell)=z^{\varphi(\ell)}_\ell$ ($\ell\in \N_0$)
which traverses a subset of $V_m$ of size bigger than $k$.
It follows that  $X^{>k}_{\tilde B} $ is uncountable, contradicting our assumption that $k$ is the essential thickness.
\end{proof}

Note that if a double path $\bar\gamma$ has thickness $k<\infty$, then
there exists $m_0$ such that also the opposite inequality is true.
More precisely, for uncountably many $\varphi\in \{0,1\}^{\N_0}$ we have
$$\exists m\, \forall n\geq m\:
| \th{m}{\varphi(\db{\gamma})}\cdots  \th{n}{\varphi(\db{\gamma})}(V_{n+1})| \geq k .$$
Indeed, the contrary, that is, the assumption that for all but at most countably many $\varphi$ we have
$$ \forall m\,\exists n\geq m\:
| \th{m}{\varphi(\db{\gamma})}\cdots  \th{n}{\varphi(\db{\gamma})}(V_{n+1})| < k, $$
implies that all but at most countably $\varphi(\db{\gamma})$ belong to $X_{\tilde B}^{<k}$, a contradiction.

We immediately obtain
\begin{cor}\label{cor-sand}
Let $(X_B,\varphi_\w )$ be a thick Toeplitz Bratteli-Vershik system with essential thickness $k$.
 Possibly after telescoping, there exists a 
choice function $\varphi$ such that  $| \th{m}{\varphi(\db{\gamma})}  (V_{m+1})|=k$ for all large enough $m$.
\end{cor}
%%%%%%%%%%%%%%%%%%%%%%%%%%%%%%%%%%%%%%%
\subsection{Toeplitz systems and non-tameness}\label{Toeplitz tameness}

In this section we prove one of our main results, Theorem \ref{Toeplitz non-tame}.  It applies to all finite rank Toeplitz systems, as stated in Theorem~\ref{Toeplitz finite rank non-tame}.

As before, we identify subsets $A_n\subset V_n$ with vertices $A_{n}\in \tilde V_n$.
For the convenience of the reader, we provide a proof of the next statement which
is reminiscent of \cite[Theorem~13.1]{Downarowicz2005}.
\begin{lem}\label{lem-density}
Consider a Toeplitz Bratteli-Vershik system with 
characteristic sequence $(\ell_n)$ and let
$(A_n)$ be some sequence of extendable vertices
$A_{n}\subset V_{n}$.
Suppose that the set of singular points of $\Z_{(\ell_n)}$ has Haar measure $0$. 

By telescoping, we can ensure that
\[ \lim_{n\to\infty} \frac{\left| \{i\in [0,\ell_n-1]: |\theta_i^{(n)}(A_{n+1})|\geq 2 \}\right|}{ \ell_n} = 0.\]
\end{lem} 
\newcommand{\head}{\mathrm{head}}
\begin{proof}
Observe that the Haar probability measure $\mu$
of the set $D\subset \Z_{\ell_n}$ of singular points is
$$\mu(D) = \lim_{n\to +\infty} \frac1{\prod_{k=0}^{n}\ell_k} |\{z_n\cdots z_0:z\in D\}|=0.$$

Let $n\geq m\in \N_0$ and $w= w_n\cdots w_m$, with $0\leq w_i \leq \ell_i-1$.
If there is an extendable $A_{n+1}\ssq V_{n+1}$ such that $|\th{m}{w_m}\cdots \th{n}{w_n}(A_{n+1})|\geq 2$,
then $w$ is a subword of some singular point $z\in D$, see Lemma~\ref{lem-x}.
 Hence, given any sequence of extendable vertices $A_{n}\in  \tilde V_{n}$, and any $m\geq 0$, we have
\begin{align*}
&\limsup_{n\to+\infty} \frac1{\prod_{k=m}^{n}\ell_k} |\{w_n\cdots w_m\:
|\th{m}{w_m}\cdots \th{n}{w_n}(A_{n+1})|\geq 2\}|\\
&\leq \limsup_{n\to+\infty} \frac1{\prod_{k=m}^{n}\ell_k} |\{z_n\cdots z_m\:
z\in D\}|\leq
\prod_{k=0}^{m-1}\ell_k\cdot \limsup_{n\to+\infty} \frac1{\prod_{k=0}^{n}\ell_k} |\{z_n\cdots z_0\:
z\in D\}|
= 0.
\end{align*}
Telescoping the from level $m$ to $n$ the statement follows.
 \end{proof}

Recall that for systems with finite topological rank,  Lemma~\ref{lem:thick-double-path} tells us that essential thickness $k>1$ is equivalent to 
the existence of a double path of thickness $k$. 
Therefore the following technical proposition applies to all finite rank Toeplitz systems with uncountably many singular points.

\begin{prop}\label{arithmetic-progression-prop}
\label{cor-loop} Let $(X_B,\varphi_\om)$ be  a thick Toeplitz Bratteli-Vershik system. 
Suppose that the set of singular points has Haar measure $0$.  Possibly after telescoping there exist
\begin{enumerate} 
\item   for any $n\geq 1$ an arithmetic progression\footnote{i.e.\ three numbers $j_0,j_1,j_2$ such that $j_{2}-j_1=j_1-j_0$} $j_0\ho{n}, j_1\ho{n}, j_2\ho{n}\in \{0,\cdots,\ell_n-1\}$, 
 and sets  $A_n\subset V_n$, 
such that $\theta\ho{n}_{j_1\ho{n}}$ and $\theta\ho{n}_{j_2\ho{n}}$ restrict to the same bijection from $A_{n+1}$ to $A_{n}$ 
while $B_{n} :={ \theta\ho{n}_{j_0\ho{n}}(V_{n+1})}$ is a proper subset of $ A_{n}$,
\item for any $n>1$ an $i_n\in \{0,\cdots,\ell_n-1\}$ such that  {$\th{n}{i_n}(V_{n+1})$} is contained in $\{a\in A_{n} : {\th{n-1}{j_1\ho{n}}}(a) \notin B_{n-1}\}$.
\end{enumerate}
\end{prop}
\begin{proof} Let $k$ be the essential thickness of the Toeplitz Bratteli-Vershik system and $\db{\gamma}$ be a double path of thickness $k>1$. It is  a sequence of parallel edges 
 $\db{\gamma}_{n} = (\gamma_{n,1},\gamma_{n,2})$, that is,
$$s(\gamma_{n,1}) = s(\gamma_{n,2}) = A_{n}\subset V_{n},\quad
r(\gamma_{n,1}) = r(\gamma_{n,2}) = A_{n+1}\subset V_{n+1}$$
such that, for large enough $n$, $|A_n|=k$, and
furthermore,  the maps
$\th{n}{\ord(\gamma_{n,1})}$and $\th{n}{\ord(\gamma_{n,2})}$
each map $A_{n+1} $ bijectively to  $A_{n}$.
Note that, by definition, all  $A_{n+1}$ are extendable.

Take the $n$-th parallel edge $(\gamma_{n,1},\gamma_{n,2})$ of our double path and set $$\Delta_{n}(\gamma_{n,1},\gamma_{n,2}) := \ord(\gamma_{n,2})-\ord(\gamma_{n,1}).$$
 Telescoping with the next level $n+1$ will produce four parallel edges.
It is crucial to observe that at least two of these four edges have the same value of $\Delta_{n}$ as the two above. Hence we can apply Lemma~\ref{lem-density} to conclude that, possibly after telescoping there is 
$j$  in $(\ord(\gamma_{n,1})+\Delta_{n}\Z) \cap \{0,\cdots,\ell_n-1\}$ 
such that  $|\th{n}{j}(A_{n+1})|<k$. 
We next need to find such a $j$ where moreover $\th{n}{j}(A_{n+1})\subset A_n$.

By Corollary \ref{cor-sand} we can find, for each $n$ large enough, $\kappa_n\in\{0,\cdots,\ell_n-1\}$ such that $\th{n}{\kappa_n}(A_{n+1}) = A_{n}$ and  {$| \th{n}{\kappa_n} (V_{n+1})|= k$, hence $\th{n}{\kappa_n} (V_{n+1})= A_{n}$.}
Define $\psi\ho{n}_j:\tilde V_{n+2}\to \tilde V_{n-1}$ through
$$\psi\ho{n}_j := \th{n-1}{\kappa_{n-1}} \th{n}{j} \th{n+1}{\kappa_{n+1}}$$
and note that $\psi\ho{n}_j(A_{n+2}) \subset A_{n-1}$. 
Moreover, the inclusion is proper if 
{$|\th{n}{j}(A_{n+1})|<k$} while $\psi\ho{n}_j$ is a bijection if $j=\ord(\gamma_{n,1})$ or $j=\ord(\gamma_{n,2})$. Thus we can find $j_0\ho{n},j_1\ho{n},j_2\ho{n} \in (\ord(\gamma_{n,1})+\Delta_{n}\Z) \cap \{0,\cdots,\ell_n-1\}$ such that $j_2\ho{n}-j_1\ho{n} = j_1\ho{n}-j_0\ho{n}$ and 
$ B_{n-1} = \psi\ho{n}_{j_0\ho{n}}(A_{n+2}) $ is a proper subset of $ A_{n-1}$ while
$\psi\ho{n}_{j_1\ho{n}}(A_{n+2}) = A_{n-1}$
  and $\psi\ho{n}_{j_2\ho{n}}(A_{n+2}) = A_{n-1}$. We now telescope the three floors together and thus the $\psi\ho{n}_{j_0\ho{n}}$,  $\psi\ho{n}_{j_1\ho{n}}$, $\psi\ho{n}_{j_2\ho{n}}$ can be realised as $\theta\ho{n}_{j_0\ho{n}}$, $\theta\ho{n}_{j_1\ho{n}}$ and $\theta\ho{n}_{j_2\ho{n}}$. 

This  shows the first statement except for the fact that 
we only know that $\theta\ho{n}_{j_1\ho{n}}$ and $\theta\ho{n}_{j_2\ho{n}}$ restrict to bijections $f\ho{n}_1$ and $f\ho{n}_2$ from $A_{n+1}$ to $A_{n}$, and it remains to show that, possibly after telescoping, $f\ho{n}_1=f\ho{n}_2$.
It is convenient to identify the $A_{n+1}$ with $A_{1}$. We do this using the bijection 
$ \tau\ho{n}:= f\ho{1}_1\cdots f\ho{n}_1 $. With
\begin{equation}\label{eq:no-tel}I_n := \tau\ho{n-1} f\ho{n}_2{\tau\ho{n}}^{-1}\end{equation}
our aim is thus to show that, possibly after telescoping, all $I_n$ are the identity.

Let further $f\ho{n}_0:A_{n+1}\to A_{n}$ be the restriction of 
$\theta\ho{n}_{j_0\ho{n}}$ to $A_{n+1}$; it is non-surjective with image $B_{n}$. 
If we telescope level $n$ with level $n+1$ we get nine compositions $f\ho{n}_i f\ho{n+1}_j$ for the three possible values of $i$ and $j$. 
Let us take a closer look at two sets of choices for them.

Consider first the maps $f\ho{n}_0 f\ho{n+1}_1$, $f\ho{n}_{1} f\ho{n+1}_1$, $f\ho{n}_2 f\ho{n+1}_1$. 
These correspond, after telescoping of the two levels, to the restriction to $A\ho{n+1}$ of maps $\theta\ho{n}_{j_0\ho{n}}$, $\theta\ho{n}_{j_1\ho{n}}$, and $\theta\ho{n}_{j_2\ho{n}}$ with $j_2\ho{n}-j_1\ho{n} = j_1\ho{n}-j_0\ho{n}$ and   $f\ho{n}_1 f\ho{n+1}_1$, 
$f\ho{n}_2 f\ho{n+1}_1$ are bijections while $f\ho{n}_0 f\ho{n+1}_1$ is not.
It is quickly seen that under this choice, the map $I_n$ after telescoping coincides with the map before telescoping. 
In other words, if we replace $f\ho{n}_2$ with  $f\ho{n}_2 f\ho{n+1}_1$ in \eqref{eq:no-tel}, then $$ \tau\ho{n-1} f\ho{n}_2   f\ho{n+1}_1{\tau\ho{n+1}}^{-1}= I_n.$$ 
 
Now consider the maps $f\ho{n}_0 f\ho{n+1}_0$, $f\ho{n}_1 f\ho{n+1}_1$, $f\ho{n}_2 f\ho{n+1}_2$. Again these correspond, after telescoping of the two levels, to the restriction to $A_{n+1}$ of maps $\theta\ho{n}_{j_0\ho{n}}$, $\theta\ho{n}_{j_1\ho{n}}$, and $\theta\ho{n}_{j_2\ho{n}}$ with $j_2\ho{n}-j_1\ho{n} =j_1\ho{n}-j_0\ho{n}$ and $f\ho{n}_1 f\ho{n+1}_1$, 
$f\ho{n}_2 f\ho{n+1}_2$ are bijections while $f\ho{n}_0 f\ho{n+1}_0$ is not.
The telescoping, however, affects the map $I_n$. The new map $\tilde I_n$ becomes 
$$ \tilde I_n = \tau\ho{n-1} f\ho{n}_2 f\ho{n+1}_2 {\tau\ho{n+1}}^{-1}=I_nI_{n+1}.$$
As $A_{1}$ is finite, the sequence $I_n$ admits a constant subsequence $g\ho{n_k} = g$. 
Telescoping the levels from $n_k$ to $n_{k+1}-1$ in the first way described above, we arrive at a situation where all $I_n$ coincide with $g$. Let $N$ be the order of $g$. 
Now, telescoping $N$ consecutive levels together in the second way above we arrive at a situation where all $I_n$ are equal to the identity. While all this telescoping has an effect on $\theta\ho{n}_{j_0\ho{n}}$,  it does not change its crucial property, namely that it maps $A_{n+1}$ to a proper subset of $A_{n}$, and that $j_0\ho{n},j_1\ho{n},j_2\ho{n}$ form an arithmetic progression.

It remains to show the second property.
Since the order $\om$ is proper,  we can assume, by telescoping if necessary, that for each $n$, $|\theta^{(n)}_0(V_{n+1})|=|\theta^{(n)}_{\ell_n -1}(V_{n+1})|= 1$. Take $a\in A_{n}$
such that ${\th{n-1}{j_1\ho{n}}}(a) \notin B_{n-1}$.
By minimality (and perhaps further telescoping), we find $i_n$ such that $\th{n}{i_n}(V_{n+1})=\{a\}$.   
\end{proof}

\begin{thm}\label{Toeplitz non-tame}
\label{thm-loop} Every  thick Toeplitz shift is non-tame.
\end{thm}

\begin{proof}
Let  $(X_B,\varphi_\om)$ be a thick Bratteli-Vershik representation of the given Toeplitz shift which we  assume to be conjugate to $(X_1,\sigma)$, see Section 2.33. 
We will construct an infinite independence set for two cylinder sets in $X_1$ which we define in the proof.

Observe that if the set of singular points in $\mc Z$ has positive Haar measure, then $(X_B, \varphi_\om)$ is non-tame \cite[Theorem 1.2]{FuhrmannGlasnerJagerOertel2018}.
We hence assume the set of singular points in $\mc Z$ to have zero Haar measure so that we can apply Proposition~\ref{arithmetic-progression-prop} in the following. 

Let $\varphi\in\{0,1\}^{\N_0}$ be a choice function.
We use the notation of the proof of  Proposition \ref{arithmetic-progression-prop} and let $h_{n}$ be the restriction of $\th{n}{i_n}$ to $A_{n+1}$. Recall that 
${f_1^{(n-1)}} h_{n}(A_{n+1})\subset A_{n-1}\backslash B_{n-1}$. 
Set
$$ z = \cdots i_{2n+2} j\ho{2n+1}_{\varphi_n}\cdots \cdots i_{2} j\ho{1}_{\varphi_0}, $$
$t_0=0$, $t_1 = (j^{(2)}_1- i_{2} ) \ell_1 + \Delta_1$, 
and for $n\geq 2$
$$t_{n} = t_{n-1} + (j^{(2n)}_1- i_{2n})\prod_{j=1}^{2n-1}\ell_{j}+ 
\Delta_{2n-1}
\prod_{j=1}^{2n-2}\ell_{j}.$$
Choose $x\in \ord^{-1}(z)$. 
By \eqref{eq-x} we have 
$x_0 \in \theta^{(1)}_{j\ho{1}_{\varphi_0}}\theta^{(2)}_{i_{2}}(A_{3})=f\ho{1}_{\varphi_0} h_{2}(A_{3})$. 
If $\varphi_0 = 0$, then 
$$f\ho{1}_{0} h_{2}(A_{3})\subset
\im f\ho{1}_{0} = B_{1}$$ whereas if $\varphi_0 = 1$ then 
$$f\ho{1}_{1} h_{2}(A_{3}) \subset A_{1}\backslash B_{1}.$$
Furthermore, $$ z + t_1 = \cdots  i_{4} j^{(3)}_{\varphi_1}j_1^{(2)} j^{(1)}_{\varphi_0+1}.$$
Hence, taking into account that $f\ho{n}_1=f\ho{n}_2$ we have
$$x_{t_1} \in \theta^{(1)}_{j^{(1)}_{\varphi_0+1}}\theta^{(2)}_{j_1^{(2)}} 
\theta^{(3)}_{j^{(3)}_{\varphi_1}}\theta^{(4)}_{i_{4}}(A_{5})
=f\ho{1}_{1} f\ho{2}_1 f\ho{3}_{\varphi_1} h_{4}(A_{5}).$$
Since
$$
f\ho{3}_{\varphi_1} h_{4}(A_{5}) = \left\{\begin{array}{ll}
f\ho{3}_{0} h_{4}(A_{5}) \subset B_{3} & \mbox{if  } \varphi_1=0,\\
f\ho{3}_{1} h_{4}(A_{5}) \subset A_{3}\backslash B_{3} & \mbox{if  } \varphi_1=1,
\end{array}\right.
$$
it follows that
$$
\begin{array}{ll} x_{t_1} \in \tau\ho{2}(B_{3}) & \mbox{ if }  \varphi_1=0,\\
x_{t_1} \in A_{1}\backslash \tau\ho{2}(B_{3}) & \mbox{ if } \varphi_1=1.
\end{array}
$$
Similarly, we find for all $n\geq 2$ that
$$
\begin{array}{ll} x_{t_n} \in \tau\ho{2n}(B_{2n+1}) & \mbox{ if } \varphi_n=0,\\
x_{t_n} \in A_{1}\backslash \tau\ho{2n}(B_{2n+1}) & \mbox{ if } \varphi_n=1.
\end{array}
$$
By finiteness of $A_{1}$ there is a subsequence of $B_{2n+1}$ such that $\tau^{2n}(B^{2n+1})$ is constant, say equal to $B$. 
Restricting to choice functions with support in this subsequence we obtain an independence sequence for the cylinder set $[B]$ and its complement in $[A_{1}]$.
\end{proof}

\begin{thm}\label{Toeplitz finite rank non-tame}
\label{thm-loop-finite} 
Let $(X,\sigma)$ be a Toeplitz shift of finite Toeplitz rank. Then  $(X,\sigma)$ is non-tame if and only if its maximal equicontinuous factor has uncountably many singular points.
\end{thm}

\begin{proof} Let $(X_B,\varphi_\om)$ be a finite rank Toeplitz Bratteli-Vershik representation of  $(X,\sigma)$. 
If the maximal equicontinuous factor of 
$(X_B,\varphi_\om)$  has uncountably many singular fibres then, by Cor.~\ref{thick-double-path} it is thick. 
By Theorem \ref{Toeplitz non-tame},  $(X_B,\varphi_\om)$ and hence $(X,\sigma)$ is non-tame.

Conversely, suppose that $(X,\sigma)$ has countably many singular fibres. 
By Lemma~\ref{lem-x} all fibres are finite.
The result follows from Corollary \ref{cor: sufficient criterion for tameness} below.
\end{proof}

\begin{remark}
In fact, the same techniques and a little more care will give us a slightly stronger version of Theorems \ref {Toeplitz non-tame} and \ref{Toeplitz finite rank non-tame}. Namely, the statements also hold true if one replaces the word non-tame by either forward non-tame or backward non-tame. What one has to ensure in these situations is that the sequence of times $(t_n)_{n\geq 0}$ is either always positive or always negative. The point in the proof where one has to be more careful is in the statement and proof of Proposition
\ref{arithmetic-progression-prop}: in the case of forward tameness, for example, one would require that $i_0<j_0<j_1<j_2$, so that each $t_n$ as defined in the proof of 
Theorem \ref{Toeplitz non-tame} is strictly positive. To achieve this requires simply further telescoping.
\end{remark}

For non-tame constant length substitution shifts with a coincidence,   an independence set can be explicitly computed, as the following example illustrates. 

\begin{myexp} [continues=ex:examplezero] We illustrate Prop.~\ref{arithmetic-progression-prop} and Theorem~\ref{Toeplitz non-tame} with our
the substitution in Example \ref{ex:first}.  The graph of Figure~\ref{thickness-not-rank-picture-1} has no parallel  edges of thickness $2$, but if we telescope two levels together we get a pair of parallel edges between the vertex $\{a,b\}$ above and the vertex $\{a,b\}$ 
below. For substitutions, telescoping amounts to taking powers of the substitution. We therefore need to work (at least) with the second power $\theta^2$ of the substitution:
\[
\begin{array}{c} a\\ b\\ c \end{array} 
\mapsto
\begin{array}{c} a\\ a\\ a \end{array}
\!\!\!\!\!\!{\color{blue} \begin{array}{c} a\\ a\\ a \end{array}}
\!\!\!\!\!\!\begin{array}{c} c\\ c\\ c \end{array}
\!\!\!\!\!\!\begin{array}{c} a\\ a\\ a \end{array}
\!\!\!\!\!  \begin{array}{c} a\\ a\\ a \end{array}
\!\!\!\!\!\!{\color{blue} \begin{array}{c} a\\ b\\ a \end{array}}
\!\!\!\!\!\!\begin{array}{c} c\\ b\\ c \end{array}
\!\!\!\!\!\!\begin{array}{c} a\\ a\\ a \end{array}
\!\!\!\!\!\begin{array}{c} a\\ a\\ a \end{array}
\!\!\!\!\!\!{\color{blue} \begin{array}{c} a\\ b\\ b \end{array}}
\!\!\!\!\!\!{\color{red} \begin{array}{c} b\\ b\\ b \end{array}}
\!\!\!\!\!\!\begin{array}{c} a\\ a\\ a \end{array}
\!\!\!\!\!\begin{array}{c} a\\ a\\ a \end{array}
\!\!\!\!\!\!\begin{array}{c} a\\ a\\ a \end{array}
\!\!\!\!\!\!\begin{array}{c} c\\ c\\ c \end{array}
\!\!\!\!\!\!\begin{array}{c} a\\ a\\ a \end{array}
\]
We see that the restrictions of ${\theta^2}_5$ and ${\theta^2}_9$ to $\{a,b\}$ are both equal to the identity. Furthermore ${\theta^2}_1$ is the projection onto letter $a$.\footnote{In the coloured version, the images of ${\theta^2}_1$, ${\theta^2}_5$, and ${\theta^2}_9$ are marked blue, that of ${\theta^2}_{10}$ is in red.} 
We may therefore chose $\theta\ho{1}=\theta^2$, $A_{2} = \{a,b\} = A_{1}$ and  
the arithmetic progression $j_0=1$, $j_1=5$, $j_2=9$. Since here the extended Bratteli diagram is stationary, all $\theta\ho{n}$ and $A_{n}$ can be taken to be equal.
As ${\theta^2}_{10}$ is the projection onto the letter $b$ we may take $i_1=10$.   
Therefore, the independence sequence from the proof of Theorem~\ref{Toeplitz non-tame}  is given by $t_0=0$, 
   and $t_{n} = t_{n-1} -5  \cdot 16^{2n-1}+ 4\cdot 16^{2n-2}$ for $n\geq 1$.
\end{myexp}

\begin{myexp}
As is the case with many properties, tameness is not preserved under strong orbit equivalence (see \cite{GjerdeJohansen} for a definition). Take the two substitutions
$$\begin{array}{c c l}
a &\mapsto  & aabaa\\ 
b & \mapsto & abbaa \\
\end{array}
$$
and
 $$\begin{array}{c c l}
a &\mapsto  & aaaba\\ 
b & \mapsto & abbaa. \\
\end{array}
$$
Then both substitution shifts are conjugate to the stationary Bratteli-Vershik system they define, see Section~\ref{sec:toeplitz-constant}. Furthermore since the underlying unordered  Bratteli diagrams are identical, these two shifts are strong orbit equivalent. However, by computing the graphs $\mathcal G_\theta$ we see that the first has one orbit of singular fibres,  whereas the second has uncountably many. In the language of this section, the second has essential thickness two. Hence by Theorem \ref{Toeplitz finite rank non-tame}, the second shift is non-tame. As the first  Toeplitz system has one singular fibre, it  is tame by Theorem~\ref{Toeplitz finite rank non-tame}.

\end{myexp}

\section{Almost automorphy and semicocycle extensions}\label{sec: semicocycle extensions}

Any almost automorphic system can be realised as a {\em semicocycle extension} of its maximal equicontinuous factor. 
Conversely, semicocycle extensions are a useful tool to construct almost automorphic systems. 
Below we will utilize this tool to construct interesting examples of Toeplitz shifts which underline the necessity of the conditions formulated in 
Theorem~\ref{Toeplitz finite rank non-tame}.

In this section, we recall the relevant details from the literature, essentially following the discussion in \cite[Section~2]{FuhrmannKwietniak2020}
where actions of general discrete groups are considered; the interested reader may also consult \cite{DownarowiczDurand2002,Downarowicz2005}.
Here, we simplify and restrict to $\Z$-actions.
We  mention that for the main purposes of this article, where we are  concerned with almost automorphic \emph{symbolic} extensions, the concept of \emph{separating covers}, as utilized in 
\cite{Paul1976,MarkleyPaul1979}, 
could also be employed.
Nevertheless, we believe that the semicocycle approach is not only more flexible but also more transparent.

Recall that $\mc Z$ is a compact metrizable monothetic group with topological generator $g$ so that $(\mc Z,+g)$ is a minimal rotation. 
Given $\hat z\in \mc Z$ and a compact metrizable space $K$, 
we call a map $f\: \hat z +\Z g \to K$ a ($K$-valued) \emph{semicocycle} over the 
pointed dynamical system
$(\mc Z,+g,\hat z)$ if it is continuous in the subspace topology on $\hat z+\Z g\ssq \mc Z$.

Given a semicocycle $f\: \hat z +\Z g \to K$ we consider the closure of its graph, 
\[
 F:=\overline{\{(z,k)\in \mc Z \times K\:z\in \hat z+\Z g,\ k=f(z)\}}\ssq \mc Z \times K.
\]
We use the letter $F$ also to denote the map
\begin{equation}\label{eq-F}
F\:\mc Z\to 2^K, \quad F(z)=\{k\in K\: (z,k)\in F\}
\end{equation}
to which we refer as the {\em section} function. Here $2^K$ denotes the set of subsets of $K$ equipped with the topology induced by the Hausdorff metric.
Since $(\mc Z,+g)$ is minimal and $K$ is compact, $F(z)$ is non-empty
for every $z\in \mc Z$. 

As we do not assume that $f$ is  uniformly continuous, its graph closure $F$ is not necessarily the graph of a function, that is, $F(z)$ may contain more than one point. In this case we call a point $z\in\mc Z$ a  \emph{discontinuity point} of $f$.
We collect the discontinuities of $f$ in the set
\begin{align*}
    D_f=\{z\in \mc Z : |F(z)|>1\}\ssq \mc Z.
\end{align*}
By definition we have
$(\hat z+ \Z g)\cap D_f=\emptyset$ and $f$ admits a continuous
extension to the complement of $D_f$, which we also denote by $f$,
\begin{align*}
f:\mc Z\setminus D_f  \to K,  \qquad f(z) = k_z,
\end{align*}
where $k_z$ is the unique point in $F(z)$.
Note that if we take any other point $\hat z\in \mc Z$ whose orbit does not contain a discontinuity  point and define $f':\hat z +\Z g \to K$ with the above extension through $f'(\hat z +ng) = f(\hat z +ng)$ then we get the same graph closure $F$ and the same set of discontinuities, $D_{f'}=D_f$.  
\begin{lem}\label{lem-new}
Let $f$ be a $K$-valued semi-cocycle for a minimal rotation $(\mc Z,+g)$ and let $F$ be the associated section function as in (\ref{eq-F}). Then $F$ is continuous on $D_f^c$. 
In particular, given $w\in D_f^c$, for any neighbourhood $W\subset K$ of $f(w)$ there exists a neighbourhood $U\subset \mc Z$ of $0$ such that for all $z'\in U+w$ we have $F(z')\subset W$. 

\end{lem}
\begin{proof} 
Choose a compatible metric $d$ on $K$ and denote by $d_H$ the associated Hausdorff metric on $2^K$. 
Let $w\in\mc D^c_f$ so that $w$ is a point of continuity of $f$, i.e.\ $F(w) = \{f(w)\}$. 
Then for any $\eps>0$ 
there is a neighbourhood $U\subset \mc Z$ of $0$ such that for all $z'\in (U+w)\cap D_f^c$ we have $d(f(z'),f(w))\leq\eps$. This implies $d_H(F(z'),F(w))\leq \eps$ first for all $z'\in (U+w)\cap D_f^c$, but then also for all $z'\in U+w$, as $F(z')\ssq \overline{\bigcup_{\tilde z\in (U+w)\cap D_f^c}F(\tilde z)}$ by definition of $F$ as the 
graph closure.
The statement follows.
\end{proof}
 
Recall that our shifts can be defined  over a compact, and not necessarily finite set; see Section~\ref{sec: basic notation}.
\begin{definition} Let $f$ be a 
 $K$-valued semicocycle  over $(\mc Z,+g,\hat z)$, and let  $X_f$ be the shift-orbit closure of the sequence $\hat f=(\hat f_n) \in K^\Z$, 
$$\hat f_n :=  f(\hat z+ ng). $$ We call $(X_f,\sigma)$   the  {\em  shift associated to $f$}.
\end{definition}
Any element of $X_f$ is hence a limit of a sequence of translates of $\hat f$.
Put differently, for each $x\in X_f$ there is $(n_k) \in \Z^\N$ such that for every $m\in \Z$ we have
$ x_m= \lim_{k\to\infty}  f(\hat z+(n_k+m)g).$
The group $\mc Z$ acts on $\mc Z\times K$ by left translation in its first factor. We say that a semicocycle $f$ over $(\mc Z,+g,\hat z)$ is \emph{separating}\footnote{Note the slight terminological deviation from 
\cite{DownarowiczDurand2002,Downarowicz2005,FuhrmannKwietniak2020} where the phrase
\emph{invariant under no rotation} is used instead of the term \emph{separating} (which we take from 
\cite{Paul1976,MarkleyPaul1979}).} if the stabiliser in $\mc Z$ of its graph $F$ is trivial, that is: if $z\in \mc Z$ satisfies $F(z+y) = F(y)$ for all $y$ then $z=0$. 
\begin{thm}\cite[Theorem 5.2]{DownarowiczDurand2002} \label{thm:general theory}
The shift associated to a separating semicocycle over $(\mc Z,+g,\hat z)$
is an almost automorphic extension of $(\mc Z,+g)$. The equicontinuous factor map 
$\pi:(X_f,\sigma)\to (\mc Z,+g)$ satisfies
\begin{align}\nonumber 
    x_n\in F(\pi(x)+ng) \quad \text{ for each } (x_n)_{n\in\Z}\in X_f \text{ and } n\in \Z.
\end{align}
\end{thm}
Accordingly, we call $(X_f,\sigma)$ a \emph{semicocycle extension of $(\mc Z,+g)$} defined by $f$. Recall that the set of singular points in $\mc Z$ of the semicocycle extension is the set of $z\in\mc Z$ such that $|\pmi(z)|>1$.  As any two distinct sequences must differ on at least one index we see that $\pmi(z)$ contains more than one point if and only if a translate of $z$ is a discontinuity point for the semicocycle. The singular points of the semicocycle extension are therefore given by $D_f+\Z g$.

The following theorem says that any almost automorphic system is conjugate to a shift over a compact alphabet, notably one obtained by a semicocycle. 
\begin{thm}[{cf. \cite[Theorem~6.4]{Downarowicz2005},\cite[Theorem~5.2]{DownarowiczDurand2002},\cite[Theorem~2.5]{FuhrmannKwietniak2020}}]
\label{thm: equivalence semicocylce extension and almost automorphic system}
Consider a topological dynamical system  $(X,T)$.
The following statements are equivalent.
\begin{enumerate}
\item\label{cond:a} $(X,T)$ is almost automorphic.
\item\label{cond:b} 
The maximal equicontinuous factor of
$(X,T)$ is a minimal rotation and $(X,T)$ is
conjugate to a semicocycle extension of this minimal rotation.
\item\label{cond:c} $(X,T)$ is conjugate to an almost automorphic shift. 
\end{enumerate}
\end{thm}
From the above, the implications (\ref{cond:b}) $\Rightarrow$ (\ref{cond:c})$\Rightarrow$ (\ref{cond:a}) are clear.
Let us briefly comment on (\ref{cond:a}) $\Rightarrow$ (\ref{cond:b})
by describing how to obtain a realisation of $(X,T)$ as a semicocycle extension.
The maximal equicontinuous factor of a minimal $\Z$-action is a minimal rotation which we denote here $(\mc Z,+g)$. Given a regular point $\hat z\in \mc Z$ with its unique pre-image $\hat x$ under $\pmax$, the function $f:\hat z+\Z g \to X$,
$$f( \hat z +ng ):= T^n(\hat x)  $$
is continuous, hence a semicocyle over $(\mc Z,+g,\hat z)$. 
The conjugacy between $(X,T)$ and $(X_f,\sigma)$ comes about as any element $x\in X$ defines a sequence in $(s_n)\in X^\Z$ through $s_n = T^n(x)$ and the sequence $\hat f$ defined by $f$ is the image of $\hat x$ under this map. 

The set of discontinuity points $D_f$ coincides with the set of singular points of $\mc Z$.
As the set of discontinuity points must be contained in the set of singular points, and here they coincide, we call the above semicocycle {\em maximal}.
\bigskip

If we are given an almost automorphic shift, then we may realise it as a semicocycle extension also in different way.
\begin{definition}\label{definition of D}
Let $(X,\sigma)\subset (K^\Z,\sigma)$ be an almost automorphic shift with maximal equicontinuous factor map $\pmax:(X,\sigma)\to (\mc Z,+g)$. We call
\begin{equation}  \nonumber
D = \{z\in\mc Z : \exists x,y\in \pmi(z) \mbox{ with } x_0\neq y_0 \}
\end{equation}
the set of \emph{discontinuity points} of the shift and the continuous map  $f^{can}:\mc Z\setminus D\to K$,
$$f^{can}(z) = x_0,\quad x\in \pi^{-1}(z)$$
the \emph{canonical semicocycle} of the shift.
\end{definition}
Indeed, if we choose a regular point $\hat z\in \mc Z$ with its unique pre-image $\hat x$ under $\pmax$, we see that 
$f^{can}$ is the continuous extension of
$\hat z +ng \to \hat x_n$ to $\mc Z\setminus D$ and hence a semicocycle over $(\mc Z,+g,\hat z)$. 
Thus $D$ is the set of discontinuity points of the canonical semicocycle.

The evaluation map $ev_0:X\to K$, $ev_0(x) = x_0$, relates
$f^{can}$ to the maximal semicocycle $f$ through $ev_0\circ f(\hat z +ng)= \hat x_n$.
The pointwise extension of $ev_0$ to $X^\Z\to K^\Z$ restricts to a conjugacy between $X_f\subset X^\Z$ and $X\subset K^\Z$. As $ev_0$ is continuous, the set of discontinuity points $D$ of the shift is contained in $D_f$, but it is often substantially smaller. 
In fact, the set $D$ is associated to the concrete realisation of the shift and is not invariant under conjugacy, as can be seen in the following example.
\begin{myexp} 
Let $\mc Z=S^1=\R/2\pi\Z$, the circle, and $g\in S^1$ be an irrational angle (incommensurate with $2\pi$). The dynamical system $(S^1,+g)$ is a minimal rotation. Let $\alpha\in\Z g$, $\alpha\neq 0$, and $\beta\in S^1\backslash\Z g$. 
The
characteristic function $\chi_{[0,\alpha)}$ on the half-open interval $[0,\alpha)$ is continuous on the $\Z$-orbit of $\beta$ and thus $f= \chi_{[0,\alpha)}$ defines a semicocycle over $(S^1,+g,\beta)$ which can easily be seen to be separating. Clearly $f$ is discontinuous at exactly two points, namely 
$D_f=\{0,\alpha\}$. Indeed, the sections $F(0)$ and $F(\alpha)$ contain two points, $\{0,1\}$, whereas all other sections contain only one point. In particular, different choices for $\alpha\in \Z g\backslash\{0\}$ lead to different semicocycle extensions, that is, different symbolic dynamical systems which, in particular, have different sets of discontinuities $D=\{0,\alpha\}$.
Note, however, that all these semicocycle extensions are topologically conjugate dynamical systems, as one can see using their description as rotations on 
 the Cantor set obtained by disconnecting the circle $S^1$ along the $\Z$-orbit of $g$
 \cite{forrest2002cohomology}.
 
We mention as an aside that the above semicocycle occurs frequently in the description of the physics of quasicrystalline condensed matter.
In particular, the function $V(n) = \chi_{[0,\alpha)}(\beta+n g)$ serves as the potential in the so-called Kohmoto model which describes the motion of a particle in a one-dimensional quasicrystal  which can be obtained from the above data $\beta$ and $\alpha$ by means of the cut \& project method \cite{KohmotoOono1984}. 
\end{myexp}

In the context of constant length substitutions we have the following description of the discontinuity points.
\begin{lem}\label{cocycle-discontinuities}
Let $\theta$ be a primitive aperiodic substitution of constant length which has a coincidence and trivial height.
The set of discontinuities $D$ of $(X_\theta, \sigma)$ coincides with the set $\hat\Sigma_\theta$ from Section~\ref{sec: substitutions}.
\end{lem}
\begin{proof}
$D^c$ is the set of points $z\in \Z_{\ell}$ for which all points in the fibre $\pi^{-1}(z)$  have the same entry at the $0$-th index. This happens if and only if {one block} of $z$'s entries, say
$z_n \dots z_k$ is such that $|\theta_{z_k}\dots \theta_{z_n}(\mathcal A)|=1$. The result follows.
\end{proof}

The corresponding result for more general Toeplitz shifts reads as follows. We leave its simple proof to the reader. Here we again assume that the given Toeplitz shift is conjugate, via $p_1$ as in \eqref{def-projection-map},
  to the top level shift $(X_1,\sigma)$ of the Bratteli-Vershik representation.
    
\begin{prop}
{Consider a Toeplitz shift with maximal equicontinuous factor $(\Z_{\ell_n},+1)$ and
Toeplitz-Brattelli-Vershik representation $(X_B,\varphi_\omega)$.}  
Then $z\in \Z_{\ell_n}$ is a discontinuity point if and only if $z=\ord(\gamma)$ for some path $\gamma=(\gamma_n)$ of the extended Bratteli diagram for which $r(\gamma_n)$ has at least $2$ elements, for all $n\geq 0$.  
\end{prop}

\section{Criteria for tameness of almost automorphic shifts}\label{sec: criteria for tameness of almost automorphic shifts}
In this part, we derive some general criteria for tameness and non-tameness of minimal
dynamical systems.
Recall that due to \cite{Huang2006,Glasner2018}, every tame minimal system is necessarily
almost automorphic.
Accordingly, we restrict to the study of almost automorphic shifts $(X,\sigma)\subset (K^\Z,\sigma)$ (or $(K^\N,\sigma)$) in this section, see also Theorem~3.3. 
We denote their maximal equicontinuous factors throughout  by 
$(\mc Z,+g)$ and the corresponding factor map by $\pi$.

To study tameness of the shift $(X,\sigma)$ we need to establish some terminology. 
Recall from Proposition \ref{prop: independence implies non-tame} that non-tameness manifests itself through the existence of a pair $V_a,V_b$ of (non-empty) disjoint compact subsets of $K$, together with an {\em independence sequence} $(t_n)$ for $V_a,V_b$, that is, a sequence $(t_n)_{n\in\N}$ in $\Z$ (or  $\N$ if we regard forward tameness)
 such that for any {\em choice function} $\varphi\in  \{a,b\}^{\N}$  we can find $x^\varphi\in X$ with $x^\varphi_{t_n} \in V_{\varphi(n)}$ for all $n\in\N$.  
Given a  choice function $\varphi$, a sequence $(t_n)$ and a pair $V_a,V_b\subset K$,
 we say that $x^\varphi$ {\em realises}  $\varphi$ along $(t_n)$ for $V_a$, $V_b$ if $x^\varphi_{t_n} \in V_{\varphi(n)}$ for all $n\in\N$. 
 Given a pair $V_a,V_b\subset K$ and a sequence $(t_n)_{n\in\N}$ with $t_n g\to z$ we  denote, for $E\subset \mc Z$,  
$$\Sigma_{(t_n)}^{V_a,V_b}(E)=\{\varphi\in  \{a,b\}^\N : \exists x^\varphi\in X\ \forall n\in\N:x^\varphi_{t_n}\in V_{\varphi(n)},{\pi(x^{\varphi})}\in E-z\},$$ 
the set of all choice functions which are realised along $(t_n)$ by points ${x^{\varphi}}\in X$ for which ${\pi(x^{\varphi})} \in E-z$. 
With this notation, $(t_n)_{n\in \N}$ is an  independence  sequence
for the pair $(V_a,V_b)$ if 
 $\Sigma_{(t_n)}^{V_a,V_b}(\mc Z) = \{a,b\}^\N$. 
 To simplify the notation we also write $\Sigma_{(t_n)}(E)$ for 
 $\Sigma_{(t_n)}^{V_a,V_b}(E)$ if the dependence on the pair $(V_a,V_b)$ is either clear from the context or one can take any pair of disjoint compact subsets of $K$. For instance, if we have a binary shift, that is, $K=\{a,b\}$, the only possible choice is $V_a=\{a\}$ and $V_b=\{b\}$, cf.\ Remark~\ref{rem: tameness of shifts with finite alphabet}.

{In the next set of results  $D$ denotes the set of discontinuities as defined in Definition~\eqref{definition of D}.}
We begin with a proposition which is implicit in the proof of \cite[Lemma~3.2]{FuhrmannKwietniak2020}.

\begin{prop}\label{prop: continuity points don't matter}
Let $(X,\sigma)\subset (K^\Z,\sigma)$ be an almost automorphic shift and let $ V_a,V_b\subset K$ be closed nonempty disjoint subsets. 
Let $(t_n)_{n\in\N}$ be a sequence in $\Z$ with $t_n g\to z$. 
The set of choice functions which are realised by  points $x\in X$ along $(t_n)$ with $\pi(x) \notin D-z$ is at most countable, that is, $\Sigma_{(t_n)}^{V_a,V_b}(D^c)$ is at most countable. 
\end{prop}
\begin{proof} 
Let $f$ be the canonical semi-cocycle and let $F$ be its associated section function.
Suppose $x$ realises the choice function $\varphi$.
All limit points of $(\sigma^{t_n}(x))$ belong to the fibre of $\pi(x)+z$ which we assume to be a point of continuity of $f$.  
Since $V_a\cap V_b=\emptyset$, there exists $i\in \{a,b\}$ such that $f(\pi(x)+z) \in V_i^c$. 
As $V_i^c$ is open, we can apply Lemma~\ref{lem-new} to $W=V_i^c$ to guarantee that there exists
a neighbourhood $U\subset \mc Z$ of $0$ such that $F(w+\pi(x)+z)\subset V_i^c$ for all $w\in U$. 
Therefore,  
if $n$ is large enough such that $t_n g-z\in U$, we have $x_{t_n}=(\sigma^{t_n}(x))_0\in F(\pi(x)+t_n)\subset V_i^c$. Hence for all large $n$, $x_{t_n} \not\in V_i$.
This means that for all large $n$, $\varphi(n)$ is constant.
The set of eventually constant choice functions is countable.
\end{proof}
As a matter of fact, we can improve Proposition~\ref{prop: continuity points don't matter}. 
To that end, we first observe
\begin{lem}\label{lem-finite}
Let $(X,\sigma)\subset (K^\Z,\sigma)$ be an almost automorphic shift and let $V_a,V_b\subset K$ be  closed nonempty disjoint subsets. 
Let $(t_n)_{n\in\N}$ be a sequence in $\Z$ with $t_n g\to z$. 
Let $w\in \mc Z$. 
If $\{t_n g+w\:n\in \N\}\cap D$ is finite, then $\Sigma_{(t_n)}^{V_a,V_b}(\{z+w\})$ is finite.
\end{lem}
\begin{proof}
By definition, $\Sigma_{(t_n)}^{V_a,V_b}(\{z+w\})$ is the set of choice functions which are realised along $(t_n)$ by points in the fibre $\pi^{-1}(w)$. If $t_ng+w\in D^c$ then all points in $ \pi^{-1}(w)$ have the same value at $t_n$. 
Since only finitely many $w+t_ng$ are discontinuity points, we can thus realise only finitely many choices through points of $\pi^{-1}(w)$.
\end{proof}

Recall that the \emph{derived set} of 
a subset of a metrizable space is the set of all its accumulation points.
We have the following strengthening of Proposition~\ref{prop: continuity points don't matter}.
\begin{prop}\label{prop: only derived set of discontinuities matters}
Let $(X,\sigma)\subset (K^\Z,\sigma)$ be an almost automorphic shift and let $ V_a,V_b\subset K$ be closed nonempty disjoint subsets. 
Let $(t_n)_{n\in\N}\subset \Z$ be a sequence
with $t_n g \to z$, and let  $D'$ denote the derived set of $D$.
Then $\Sigma_{(t_n)}^{V_a,V_b}({D'}^c)$ is at most countable.
\end{prop}
\begin{proof}
The answer to the question of whether $\Sigma_{(t_n)}^{V_a,V_b}({D'}^c)$ is at most countable or not does not change if we restrict to the subsequence of elements of $(t_n)_{n\in\N}$ which are pairwise distinct. We may therefore assume this 
without loss of generality.
Let $w\in \mc Z$ be such that 
$z+w\in D\setminus {D'}$, so that $z+w$ is an isolated point of $D$. Hence $t_n g +w\notin D$ for sufficiently large $n$, that is, $\{t_n g+w\:n\in \N\}\cap D$ is finite. 
 Moreover, since $X$ is compact, the set $D\setminus D'$ is at most countable.
 By Lemma~\ref{lem-finite} $\Sigma_{(t_n)}^{V_a,V_b}(D\setminus {D'})$ is at most countable.
 The result follows with Proposition~\ref{prop: continuity points don't matter}.
\end{proof}

In a similar way, we obtain the following useful criterion for forward tameness.
For tameness of the full $\Z$-action it has to be tested for the forward and the backward motion.  
\begin{cor}[{\cite[Lemma~3.2]{FuhrmannKwietniak2020}}]
\label{cor: sufficient criterion for tameness}
Let $(X,\sigma)$ be an almost automorphic shift with maximal equicontinuous factor $(\mc Z,+g)$. 
 Suppose that the set discontinuity points $D\subset \mc Z$  is countable.
 If $z+ \N g \cap D$ is finite for each $z\in \mc Z$, then $(X,\sigma)$ is forward tame.
\end{cor}

\begin{prop}\label{prop: co-countable implies all sequences}
Let $(X,\sigma)\subset (K^\Z,\sigma)$ be an almost automorphic shift and $ V_a,V_b\subset K$ be closed nonempty disjoint subsets. 
Let $(t_n)_{n\in\N}$ be a sequence in $\Z$ with $t_n g\to z$. 
Given $E\subset \mc Z$ we have 
$\overline{\Sigma_{(t_n)}^{V_a,V_b}(E)}\subset \Sigma_{(t_n)}^{V_a,V_b}(\overline E)$.
\end{prop}

\begin{proof} 
Let $\varphi\in \{a,b\}^\N$ be the limit of the sequence $(\varphi\ho{m})_{m}\subset \Sigma_{(t_n)}^{V_a,V_b}(E)$. 
For each $m$ choose $x\ho{m}\in E-z$ which realises $\varphi\ho{m}$. 
Let $x$ be a limit of some subsequence of $(x\ho{m})_m$. We claim that $x$ realises $\varphi$. Indeed, given $n$ there exists $m_n$ such that $\varphi\ho{m}(n)=\varphi(n)$ for $m\geq m_n$ and hence $x\ho{m}_{t_n}\in V_{\varphi(n)}$ for $m\geq m_n$.  
Since the convergence of $x\ho{m}$ to $x$ is pointwise and $V_a$ and $V_b$ are compact, this
implies that $x_{t_n}\in V_{\varphi(n)}$ for each $n\in \N$.
\end{proof}
The above results lead to the following necessary condition for non-tameness. 
\begin{cor}\label{cor: realization only needs Df}
Let $(X,\sigma)\subset (K^\Z,\sigma)$ be an almost automorphic shift and $ V_a,V_b\subset K$ be closed nonempty disjoint subsets.
If $(t_n)_{n\in\N}$ is an independence sequence for $(V_a,V_b)$ with $t_n g\to z$, then $\Sigma_{(t_n)}^{V_a,V_b}(D)=\Sigma_{(t_n)}^{V_a,V_b}(D')=\{a,b\}^\N$.
\end{cor}
\begin{proof} As $(t_n)_n$ is an independence sequence its elements are pairwise distinct and so we may apply 
Propositions~\ref{prop: continuity points don't matter}, \ref{prop: only derived set of discontinuities matters} 
to see that $\Sigma_{(t_n)}(D)$ and $\Sigma_{(t_n)}(D')$ are dense in $\{a,b\}^\N$. Since $D$ and $D'$ are closed the result follows from Proposition~\ref{prop: co-countable implies all sequences}.
\end{proof}

%%%%%%%%%%%%%%%%%%%%%%%%%%%%%%%%%%%%%%%%%%%
%%%%%%%%% Proof of Johannes using the section function %%%%%%%%%%%
\begin{lem}\label{lem: existence 0-independence set}
Let $(X,\sigma)\subset (K^\Z,\sigma)$ be an almost automorphic shift and $ V_a,V_b\subset K$ be closed nonempty disjoint subsets.
Let  $\tilde V_a$, $\tilde V_b$ be closed disjoint subsets of $K$ such that $V_i \subset \inte \tilde V_i$.
Let $(t_n)_{n\in\N}\subset \Z$  be such that $t_n g\to z$. 
Let $E\subset \mc Z$ be compact with $(E-z+t_n g)\subset {D}^c$ for all $n$.  
Then there is a sequence $(\tau_n)_{n\in\N}\subset\Z$ with
$\tau_n g\to 0$ and $\Sigma_{(\tau_n)}^{\tilde V_a,\tilde V_b}(E)\supseteq \Sigma_{(t_n)}^{V_a,V_b}(E)$. Moreover, $(\tau_n)_{n\in \N}$ can be chosen strictly negative and it can also be chosen strictly positive.
\end{lem}
\begin{proof} 
Let $f:D^c\to K$ be the canonical semicocycle and $F:\mc Z\to 2^K$ its associated section function. We apply Lemma~\ref{lem-new} to $ \inte \tilde V_i$ to obtain that, given $w\in D^c$ such that $f(w)\in  \mathrm{int} \tilde V_i$ there exists a neighbourhood $U$ of $0\in\mc Z$ such that for all $z'\in U+w$ we have $F(z')\subset  \mathrm{int} \tilde V_i$.
Let $A\subset D^c$ be a compact subset. As $f$ is continuous on $D^c$, also $A_i:=f^{-1}(V_i)\cap A$ is compact for $i=a,b$. This implies that we can choose the neighbourhood $U$ of $0\in\mc Z$ uniformly for all $w\in A$ so that, for all $z'\in U+A_i$ and  $i=a,b$, we have $F(z')\subset \mathrm{int} \tilde V_i$.

Applying the above to $A= E-z+t_n g$, with $n\in\N$, we conclude  that there exists a neighbourhood $U_n$ of $0\in\mc Z$ such that 
\begin{equation}\label{eq-F1}
F(z')\subset \tilde V_{i}\quad \mbox{\rm for all}\; z'\in U_n + \big((E-z+t_n g)\cap f^{-1}(V_i)\big).
\end{equation}
As the dynamics on $\mc Z$ is forward minimal there exist $t'_n\in \Z^+$ such that 
$t'_{n}g\in -U_n+z$.
Define $\tau_n = t_n-t'_{n}$. Without loss of generality we may assume that $\limsup_n U_n = \{0\}$ 
so that $t'_n g\to z$. Hence  $\tau_n g$ tends to $0\in\mc Z$ as $n\to +\infty$.

Let $\varphi \in \Sigma_{(t_n)}^{V_a,V_b}(E)$. 
There is $x\in \pi^{-1}(E-z)$ realising $\varphi$ along $(t_n)$. 
 Let $w = \pi(x) +z$. As $f(w-z+t_ng) = x_{t_n}\in V_{\varphi(n)}$ we have
 \begin{equation}\label{eq-F2}
F(z')\subset \tilde V_{\varphi(n)}\quad \mbox{\rm for all}\; z'\in U_n +w-z+t_ng
\end{equation}
with $U_n$ as in (\ref{eq-F1}).

Let  $y\in\pi^{-1}(w)$. Then $y_{\tau_n} = (\sigma^{\tau_n}(y))_0 \in F(w+\tau_n g)$. 
We have 
$$w+\tau_n g = w-t'_n g+t_n g \in U_n+w-z+t_n g$$  
which implies---by (\ref{eq-F2})---that 
$y_{\tau_n}\in F(w+\tau_ng)\subset  \tilde V_{\varphi(n)}$. 
We thus have shown that, for all $n$, $y_{\tau_n}\in \tilde V_{\varphi(n)}$. This means that $\varphi$ is realised by $y$ along $(\tau_n)$ for $\tilde V_a,\tilde V_b$, or, in other words $\varphi\in \Sigma_{(\tau_n)}^{\tilde V_a,\tilde V_b}(E)$.

For the last part, note that we can choose $t'_n$ to rise fast enough such that $t_n-t'_n\to -\infty$. Or we can use backward minimality and  choose $t'_n\in \Z^-$ to fall fast enough so  that $t_n-t'_n\to +\infty$.
\end{proof}

For later reference, we apply the above results to binary shifts, that is $K=\{a,b\}$.
Note that in this case, we are bound to choose $V_a=\tilde V_a=\{a\}$ and $V_b=\tilde V_b=\{b\}$ in Lemma~\ref{lem: existence 0-independence set}.

\begin{thm}\label{thm-reduction}
Let $(X,\sigma)\subset (\{a,b\}^\Z,\sigma)$ be a non-tame almost automorphic shift with maximal equicontinuous factor $(\mc Z,+g)$. 
\begin{enumerate}
\item There is an independence set $(t_n)$ for the pair $\{a\},\{b\}$ with $t_n g\to z$ such that all choice functions are realised along $(t_n)$ by elements of $D'-z$.
\item If for each $z\in \mc Z$ there are only finitely many $t\in\Z$ such that $(D-z+t g)\cap D \neq \emptyset$, then there exists an independence set $(t_n)$ with $t_ng\to 0$
such that all choice functions are realised along $(t_n)$ by elements of $D$. Moreover, $(t_n)$ can be chosen strictly negative and it can also be chosen strictly positive.
\end{enumerate}
\end{thm}
\begin{proof} 
The first statement is Prop.~\ref{prop: only derived set of discontinuities matters} applied to the only possible pair of non-empty disjoint subsets of $\{a,b\}$. As for the second, let $(t_n)\subset \N$ be an independence set such that $t_n g\to z$.  
By Corollary~\ref{cor: realization only needs Df} all choice functions are realised along $(t_n)$ by points of $\pi^{-1}(D-z)$. Taking a subsequence of $(t_n)\subset \N$ we can suppose that $(D-z+t_n g)\cap D\neq \emptyset$ for all $n$. As $K$ is finite, all its subsets are clopen and so we may apply Lemma~\ref{lem: existence 0-independence set} to find an independence set $(\tau_n)$ with $\tau_n g\to 0$ 
such that all choice functions are realised along $(\tau_n)$ by elements of $D$.
We may choose $(\tau_n)$ strictly negative or strictly positive.
\end{proof}

\section{A tame Toeplitz shift with uncountably many singular fibres}\label{sec: a tame toeplitz shift with uncountably many singular fibres}
In this section we construct an almost automorphic tame extension of the odometer 
$\mc Z:=\Z_{(4^n)}$ by means of a semicocycle  whose discontinuity set $D$ is uncountable. 
Note that  Theorem \ref{Toeplitz finite rank non-tame} implies that the Toeplitz shift we construct here does not have finite Toeplitz  rank.

\subsection{Set of discontinuities $D$ and its properties}

Given $z\in\mc Z=\Z_{(4^n)}$, we denote by $\head_i(z) = z_i \ldots z_1$ the \emph{head} of length $i$. 
We use the \emph{common head length} function
$$L(z,z') := \sup\{ i\in \N\: \head_i(z)=\head_i(z')\}$$
and further set $L(z,A) = \sup\{L(z,z')\: z'\in A\}$ for $A\subset \mc Z$.
Recall that the topology of $\Z_{(4^n)}$ has a base of clopen sets $U_w$ labelled by finite words $w=w_k\cdots w_1$, $w_i\in  \{0,\cdots, 4^i-1\}$, where
$$U_w = \{z\in \Z_{(4^n)} : \head_k(z)=w\}.$$
In the following, integers $t\in \Z$ are identified with their natural representation in $\mc Z$ and we write $z_n(t)$ for the $n$-th entry in that representation.

We recursively define a chain of subsets $D^i$ of $\mc Z$. 
The subset $D^i$ will contain $m_i:=2^{i}$ elements which are all of the form $z=\cdots z_2 z_1$ with $z_i = 3^{k_i}$ a power of $3$. 
We write $\bar{\cdot}$ to indicate infinite repetition to the left.
\begin{itemize}
\item[(0)] $D^0=\{\overline{3^0}\}$. 
\item[($\tilde 0$)] $\tilde D^0 = \{\overline{3^1} 3^0\}$. 
\item[(1)] $D^1 = D^0\cup \tilde D^0 = \{\overline{3^0},\overline{3^1} 3^0\}$. 
\item[(i+1)] Suppose that we have constructed $D^{i}$ and ordered its elements (in some fashion)
\[D^{i} = \{d^{(i,0)},\cdots,d^{(i,m_i-1)}\}.\]
For each $l\in\{0,\ldots,m_i-1\}$, we define ${\tilde d^{(i,l)}\in}\tilde D^{i}$ to be
\begin{align*}
\tilde d^{(i,l)} = \overline{3^{m_i+l}}\; \head_{m_i+l} d^{(i,l)}
\end{align*}
and $D^{i+1} := D^{i}\cup \tilde D^{i}$.
\end{itemize}

We may order $D^{i+1}$ in a way where we first take the elements of $D^{i}$ and then those of $\tilde D^{i}$.
For example, this gives 
\begin{align*}
{D^2} =&\{\overline{3^0},\overline{3^1}3^0,\overline{3^2}3^03^0,\overline{3^3}3^13^13^0\}
\intertext{and}
  {D^3} =&\{\overline{3^0},\,\,\overline{3^1}3^0,\,\,\overline{3^2}3^03^0,\,\,\overline{3^3}3^13^13^0\} \\
  &\cup
\{\overline{3^4}3^03^03^03^0,\,\, \overline{3^5}3^13^13^13^13^0,\,\,\overline{3^6}3^23^23^23^23^03^0,\,\,\overline{3^7}3^33^33^33^33^13^13^0 \}.
 \end{align*}
Define
\[
 D=\overline{\bigcup_{i\in \N} D^i}.
\]

We denote by 
$\mathrm{Head}_m:=\{\head_m(z) \: z\in D\}$, the heads of length $m$ of elements of $D$ and by
$q:\mathrm{Head}_{m+1}\to \mathrm{Head}_m$ the obvious surjective map given by shortening the head by one element. 
A word $w\in \mathrm{Head}_m$ is {\em (left)  special} if it does not have a unique preimage under $q$.
\begin{prop}\label{prop: essential properties of D}
The set $D$ satisfies the following properties:
\begin{enumerate}
\item[P1]\label{P1} 
Let $z, z'\in D$. 
If $z_n = z'_n$, then $\head_n(z) = \head_n(z')$.
\item[P2] \label{P2} 
Any $\Z$-orbit {in $\Z_{(4^n)}$} hits $D$ at most once. Moreover, no $d\in D$ satisfies $d_n=0$ or $d_n=4^n-1$ for any $n$.
\item[P3]\label{special} 
For each $m$, there is a unique special word $w_m \in \mathrm{Head}_m$.
\item[P4] $D$ is uncountable.
\end{enumerate}
\end{prop}

\begin{proof}
P1 is established by direct inspection of the sets $D^i$.

As for the first part of P2, note that distinct elements of $D$ cannot be tail equivalent, due to P1. 
The second part is immediate.

To see P3, notice that $\mathrm{Head}_m$ has exactly $m$ elements, namely $\mathrm{Head}_{2^i+l}$ consists of the first $2^i+l$ elements of $D^{i+1}$, $l=0,\cdots, 2^i-1$. 
Hence $q$ is $1$-$1$ on all but one element.

Finally, P4 follows since no point of $D^i$ remains isolated (in $D$) so that
$D$ is a non-empty compact set without isolated points. 
Such a set is uncountable.
\end{proof}
%%%%%%%%%

\begin{lem}\label{lem: D+z is disjoint from D}
 For all $z\in \mc Z\setminus \Z$, there is at most one $d\in D$ and one $t\in \Z$ such that
 $d+t+z\in D$. In particular, for all $z\in \mc Z$, there is at most one $t\in\Z$ such that
 $(D+z+t)\cap D\neq\emptyset$.  
\end{lem}
\begin{proof} Let $z\in \mc Z\setminus \Z$. 
Suppose first that the condition $d+t+z \in D$ and $d+t'+z \in D$ can be satisfied by one $d\in D$ but perhaps two distinct $t,t'$. Then the orbit of $d+t+z$ intersects $D$ twice which contradicts Property~P2.
 
Now suppose there are distinct $d$ and $d'$ in $D$ 
and $t,t'\in \Z$ with both $d+t+z$ and $d'+t'+z\in D$.
Any $d\in D$ is a sequence of powers of $3$, $d=(3^{p_n})_{n\in \N}$
with $p_n\leq n-1$ and hence 
\begin{align}\label{eq: maximal coordinate}
d_n\leq 3^{n-1} \quad\text { for each } n\in \N \text{ and all } d\in D. 
\end{align}
We write
 \[d=(3^{p_n})_{n\in \N}, \,\, d'=(3^{p'_n})_{n\in \N}, \,\,d+t+z=(3^{q_n}),\,\,
 d'+t'+z=(3^{q_n'})\, .\]
Suppose first that $t,t'\geq 0$. Then there is $M$ such that for all $n\geq M$ we have 
$z_{n-1}(t)=z_{n-1}(t')=0$ (recall that $z_n(t)$ is the $n$-th entry of $t$ understood as an element of $\mc Z$). 
Let $c_n$ and $c'_n$ be the carry over from the $n-1$-th 
to the $n$-th coordinate in the addition of
$d$, $t$, $z$ and the addition of $d'$, $t'$, $z'$, respectively. We claim that $c_n=c'_n$ if $n\geq M$. 
Suppose for a contradiction that $c_n=1$ while $c'_n=0$. 
Then $c_n=1$ means that $d_{n-1}+z_{n-1}(t)+z_{n-1}+c_{n-1}\geq 4^{n-1}$ 
while $c'_n=0$ means that $d'_{n-1}+z_{n-1}(t')+z_{n-1}+c'_{n-1} = 3^{q'_{n-1}}$.
The first inequality together with \eqref{eq: maximal coordinate}
implies $z_{n-1}\geq 4^{n-1}-3^{n-2}-1$. We thus get
\[
3^{q_{n-1}'}=d_{n-1}'+z_{n-1}(t')+z_{n-1}+c_{n-1}'\geq d_{n-1}'+4^{n-1}-3^{n-2}-1>3^{n-2}
\]
which contradicts the assumption that $(3^{q'_n})\in D$ because of \eqref{eq: maximal coordinate}.
It follows that, indeed, $c_n=c'_n$ if $n\geq M$.

Since $z\not\in \Z$ we have $d\neq d+t+z$ for all $t\in\Z$.  
As also $d\neq d'$, 
Property~P1 implies that there is $M'$ such that $p_n\neq q_n$ and  $p_n\neq p'_n$
for all $n\geq M'$. 
Choose $n$ larger than $M$ and $M'$. 
As $z_{n}(t)=z_{n}(t')=0$, we get 
$ 3^{q_n} = d_n+z_n+c_n = 3^{p_n} + z_n + c_n$ and likewise
$ 3^{q'_n} = 3^{p'_n} + z_n + c'_n$. Using $c_n=c'_n$ this gives
\[
 3^{p_n}+3^{q_n'}=3^{p_n'}+3^{q_n}
\]
which contradicts the assumption that $p_n\notin \{q_n, p_n'\}$.
Hence, the first statement follows for $t,t'\geq 0$. 
Now if $s=\min\{t,t'\}<0$, we replace $z$ in the above argument by $z-s$ to conclude. 

To prove the second statement, suppose that  $(D+z+t)\cap D\neq\emptyset$. This means that there exists $d\in D$ such that $d+z+t\in D$. 
If $z\in\Z$, then Property~P2
implies $t=-z$. Otherwise $d$ and $t$ are uniquely determined by $z$ according to the first statement. In particular $(D+z+t')\cap D\neq\emptyset$ implies $t'=t$. 
\end{proof}

%%%%%%%%%%%%%%%%%%%%%%%%%
\subsection{The semicocycle and its extension}
We next turn to the construction and discussion of the desired semicocycle extension.
Define $ f\: \mc Z\setminus D \to \{a,b\}$ by
\begin{equation}\label{eq-semi-I}
 f(z) =
 \begin{cases}
  a & \text{if  $L(z,D)$ is odd}\\
  b & \text{otherwise}.
 \end{cases}
\end{equation}
In other words, $f(z)=a$ if the largest possible head-overlap an element $d\in D$ can have with $z$ has odd length.

\begin{lem}
The function $f$ is continuous.
Accordingly, given a point $z\in \mc Z$ with $z+\Z\cap D=\emptyset$, we have
that $f$ is a semicocycle over $(\mc Z,+1,z)$.

Furthermore, the set of discontinuities of this semicocycle coincides with $D$ so that $f$ is separating. 
\end{lem}

\begin{proof}
Given $w\in \mc Z$, the function $z\mapsto L(z,w)$ is continuous on 
$\mc Z\setminus \{w\}$ and cannot be continuously extended to all of $\mc Z$. 
This gives that $f$ is continuous and likewise that $D$ 
is the set of discontinuities of the respective semicocycle over $(\mc Z,+1,z)$ for any appropriate $z$.

To prove that $f$ is separating, let $d, d'\in D$ and assume that $D+z=D$ with $z\notin\Z$. 
Then $d+z\in D$ and $d'+z\in D$ which by Lemma~\ref{lem: D+z is disjoint from D} 
implies $d=d'$. But $D$ is not a singleton, a contradiction.
By the same lemma $(D+t)\cap D \neq \emptyset$ with $t\in \Z$ implies $t=0$.
Hence the stabiliser in $\Z_{(4^n)}$ of $D$ is trivial.
\end{proof}
As dicussed in Section~\ref{sec: semicocycle extensions}, we hence obtain a semicocycle extension $(X_f,\sigma)$ of $(\mc Z,+1)$.
 
\subsection{Absence of independence}

By Remark~\ref{rem: tameness of shifts with finite alphabet}, $(X_f,\sigma)$ is forward tame if the cylinder sets $[a]$, $[b]$ do not have an infinite independence set in $\N$. This is what we now show.

As before, given $t\in \N$, we let $\ldots z_2(t)z_1(t)$ denote its representation in $\mc Z$.
Define
\begin{align*}
 m(t)&=\min\{n\in \N\: z_n(t)\neq 0\},\\
M(t) & = \min\{n\in \N\:  z_{m}(t)=0 \text{ for all } m\geq n \}.
 \end{align*}
\begin{lem}\label{lem:head-different}
Let $t\in \N$ and $z\in D$. For all $z'\in D$ we have $\head_{M(t)+1}(z+t) \neq \head_{M(t)+1}(z')$.
In particular, $L(z+t,D)\leq M(t)$.
\end{lem}
\begin{proof}
When adding $t$ to $z$, denote by $c_l$ the carry over from the $l-1$-st entry to the $l$-th entry. These carry overs are determined successively, starting with $l=2$. However, whatever the carry over $c_{M(t)}$ is, as $z_{M(t)}(t)=0$ and $z_{M(t)}<4^{M(t)}-1$, 
the carry over
 $c_{M(t)+1}$ must be $0$. Hence $z_l = (z+t)_l$ for all $l\geq M(t)+1$.
 
Suppose for a contradiction that $\head_{M(t)+1}(z+t) = \head_{M(t)+1}(z')$ for some $z'\in D$. Then, by the above, $z_{M(t)+1} = (z+t)_{M(t)+1} = z'_{M(t)+1}$. As $z,z'\in D$ this implies that $\head_{M(t)+1}(z) = \head_{M(t)+1}(z')$ (Property~P1). Hence, $\head_{M(t)+1}(z) = \head_{M(t)+1}(z+t)$ so that $z_l = (z+t)_l$ for all $l\leq M(t)+1$. This implies $z+t=z$, a contradiction as $t>0$.
\end{proof}
\begin{cor}\label{cor-head1}
Let $t\in \N$ and $z,z'\in D$. If $\head_{M(t)}(z)=\head_{M(t)}(z')$ then $L(z+t,D) = L(z'+t,D)$.
\end{cor}
\begin{proof}
Suppose first that $L(z+t,D)<M(t)$. 
Then for all $z''\in D$ there is $n< M(t)$ with $z''_n \neq (z+t)_n$.  
As $\head_{M(t)}(z+t)= \head_{M(t)}(z'+t)$, this implies $L(z+t,z'') = L(z'+t,z'')$
and thus $L(z+t,D)=L(z'+t,D)$.
Likewise $L(z'+t,D)<M(t)$ implies $L(z+t,D)=L(z'+t,D)$. 
By Lemma \ref{lem:head-different}, the only other possibility is $L(z+t,D)=L(z'+t,D)=M(t)$.
\end{proof}

%%%%%%%%%%%%%%%%%%%%%%%
\begin{prop}\label{lem-head1}
Suppose $0<t_1<t_2$ and $M(t_1)<m(t_{2})$. 
Then there exists a choice function $\varphi\: \{1,2\}\to \{a,b\}$ which cannot be realised along $t_1,t_2$ by elements $x\in X_f$ with $\pi(x)\in D$, that is, there is no $x\in \pi^{-1}(D)$ with $x_{t_i}\in [\varphi(i)]$ for $i=1,2$ where $\pi$ denotes the factor map onto $(\mc Z,+1)$.
\end{prop}
\begin{proof}
We assume that $m(t_2)$ is even; the other case has a similar proof.

Let $z\in D$ and suppose that $w := \head_{m(t_2)-1}(z)$ is not the special element of $\mathrm{Head}_{m(t_2)-1}$. Let $z'\in D$. If $L(z,z')< m(t_2)-1$ then also  $L(z+t_2,z')< m(t_2)-1$ (as $\head_{m(t_2)-1}(z)=\head_{m(t_2)-1}(z+t_2)$). On the other hand, if 
$L(z,z')\geq m(t_2)-1$ then we must even have $L(z,z')\geq m(t_2)$. Indeed, 
as $w$ is not special, $\head_{m(t_2)-1}(z)=\head_{m(t_2)-1}(z')$ implies $\head_{m(t_2)}(z)=\head_{m(t_2)}(z')$. Now  
$L(z,z')\geq m(t_2)$ implies that $z'_{m(t_2)}\neq (z+t_2)_{m(t_2)}$ and hence
$L(z+t_2,z') \leq m(t_2)-1$. Hence the largest possible head-overlap an element of $D$ can have with $z+t_2$ is $m(t_2)-1$. 
On the other hand this overlap can be obtained by using $z'=z$, as $L(z,z+t_2)=m(t_2)-1$. 
Since  $m(t_2)-1$ is an odd number, we have $f(z+t_2)=a$ (observe that $f(z+t_2)$ is well-defined
as $z+t_2\notin D$ due to P2).
Therefore $f(z+t_2)=b$ implies that  $\head_{m(t_2)-1}(z) = w_{m(t_2)-1}$, the special element of $\mathrm{Head}_{m(t_2)-1}$.

Let $z,z'\in D$. Suppose that $f(z+t_2) = f(z'+t_2) = b$. We just saw that this implies that 
the heads of length $m(t_2)-1$ of
$z$ and $z'$ are both equal to the special element $w_{m(t_2)-1}$ so that, in particular, 
$\head_{m(t_2)-1}(z)= \head_{m(t_2)-1}(z')$. As $M(t_1)<m(t_2)$ this implies $\head_{M(t_1)}(z)=\head_{M(t_1)}( z')$
and hence, by Cor.~\ref{cor-head1},  
$f(z+t_1) = f(z'+t_1)$. 
This means that we have no choice for the first element: One of the two choice functions $\varphi(1)=a$, $\varphi(2)=b$, or $\varphi'(1)=b$, $\varphi'(2)=b$
cannot be realised along $t_1,t_2$ by elements of $\pi^{-1}(D)$.
\end{proof}

\begin{thm}\label{thm:tame-toeplitz}
 $(X_f,\sigma)$ is tame.
\end{thm}
\begin{proof}
Suppose that $(X_f,\sigma)$ is non-tame.
Lemma~\ref{lem: D+z is disjoint from D} guarantees that the assumptions of the second part of   Theorem~\ref{thm-reduction} are satisfied so that we can then find an independence set $(\tau_n)\subset \N$ such that $\tau_n\to 0$ in $\mc Z$ and every choice function is realised 
along $(\tau_n)$ by elements of $D$. As $\tau_n\to 0$ we have $m(\tau_n) \to +\infty$. Hence we can find two elements $t_1 = \tau_{n_1}$ and $t_2 = \tau_{n_2}$ which satisfy $M(t_1)<m(t_{2})$. Proposition~\ref{lem-head1} says that not all choice functions for $(t_1,t_2)$ can be realised by elements of $D$. This contradicts the fact that all choice functions can be realised along $(\tau_n)$ by elements of $D$. Hence $(X_f,\sigma)$ is tame.
\end{proof}

\section{Toeplitz shifts with a unique singular orbit}
\label{sec: Toeplitz shifts with finitely many singular fibres}
In this section, we study a class of Toeplitz shifts 
which have a single orbit of singular fibres. 
While the set of singular points in the maximal equicontinuous factor is thus countable, the singular fibres themselves will be uncountable and this will give rise to the possibility that the shift is non-tame. 
In fact, this was already observed in \cite{FuhrmannKwietniak2020}.
However, by refining the construction from \cite{FuhrmannKwietniak2020}, we will show that
an uncountable singular fibre does not ensure non-tameness, just as uncountably many singular orbits don't ensure non-tameness, as we saw in Theorem \ref{thm:tame-toeplitz}.
Further, we will see  that minimal non-tame systems can still be forward (or likewise backward) tame.

Specifically, given any language $\mathcal L$ on a two-letter alphabet, 
we construct a binary Toeplitz shift $(X,\sigma)$, whose singular points are $\Z$, and  such that there is a positive sequence $(t_n)$ of integers with $\Sigma_{(t_n)}(\{0\})=\mathcal L$.  
Then, taking $\mathcal L$ to be the language of a Sturmian shift, or alternatively 
$\{a,b\}^{\N}$, we find that $X$ can be either forward tame, or not.
Independently of $\mc L$, however, all of the constructed systems will be backward tame.
\subsection{The semicocycle extension}
Consider a double sequence $(l^n_i)_{n\in \N,i\in\N_0}\subset \N_0$ with the following properties
\begin{itemize}
 \item[(R1)] \label{cond:R1} $(l_0^n)_{n\in \N}$ is strictly increasing.
 \item[(R2)] \label{cond:R2} For all $n\in \N$, $(l^n_i)_{i\in\N_0}$ is  strictly increasing.
 \item[(R3)] \label{cond:R3} For all $n\in\N$ there is $i_{n}\in \N_0$ such that $l^n_{i+i_{n}}=l^{n+1}_{2i}$.
 \end{itemize}
An example is given by
$$ l^1_i = 2^i-1,\quad l^{n+1}_{2i} = l^n_{i+2^{n-1}},\quad 
   l^{n+1}_{2i+1}  = \frac12(l^n_{i+2^{n-1}} + l^n_{i+1+2^{n-1}})  $$
with first values  
$$\begin{array}{cccccccccccc}
0 & 1 & 2 & 3 & 4 & 5 & 6 & 7 & 8 & 9 & 10 & 11\\
\hline
l^1_0 & l^1_1 &  & l^1_2  & & & & l^1_3 & & & & \\
& l^2_0 & l^2_1 & l^2_2 & &  l^2_3  & & l^2_4 & & & & l^2_5 \\
& & & l^3_0 & l^3_1 & l^3_2 &   l^3_3  &  l^3_4 & & l^3_5 & & l^3_6\\
& & & & & & & l^4_0 & l^4_1 & l^4_2 &   l^4_3  &  l^4_4 \\
\end{array} 
$$
Consider a right-extendable language $\mc L=\bigcup_{n\in \N} \mc L^n\subset \{a,b\}^\ast$ over the alphabet $\{a,b\}$ where $\mc L^n =\mc L\cap \{a,b\}^n$ and $\mc L^1=\{a,b\}$ (i.e., $\mc L$ is not trivial). 
Right-extendable means that any word $w\in \mc L^n$ has an extension to the right
$wc\in \mc L^{n+1}$ where $c\in \{a,b\}$.  
Similarly as in the previous section, if $w$ has two extensions in $\mc L^{n+1}$ it is called \emph{right special}. We denote the set $\{n,n+1,\cdots,m-1\}$  by $[n,m)$. 
\begin{lem}\label{lem-G1}
Let $(l^n_i)_{n\in \N,i\in\N_0}$ satisfy the conditions (R1)--(R3) above, and let $\mathcal L$ be a right-extendable binary language. 
Then there exists a family $(f^n)_{n\in \N}$ of functions $f^n:\N_0\to \{a,b\}$ which satisfy
the following conditions for all $N\in\N$.
\begin{enumerate}
 \item \label{item: defn w Ni}
 For each $i\in \N_0$ the function 
 $$ x\mapsto  f^1(x)f^2(x)\ldots f^N(x) \in \{a,b\}^{N}$$ is constant 
 on $[l^N_{i}, l^N_{i+1})$ and the constant value is a word $w^{(N,i)}\in \mc L^N$. 
 \item \label{item: L2} Each word $w\in \mc L^N$ arises infinitely often in the above way, that is, for each $i_0\in \N$ there is $i\geq i_0$ with 
 $w=w^{(N,i)}:=f^1(x)f^2(x)\ldots f^N(x)$ for $x\in [l^N_{i}, l^N_{i+1})$.
\end{enumerate}
\end{lem}
\begin{proof}
Let $f^1$ be any function which takes infinitely often each of the values $a$ and  $b$ and is constant on all $[l^1_{i}, l^1_{i+1})$.
 Then the above is satisfied for $N=1$.

Now, suppose we have already defined $f^1,\ldots,f^N$ satisfying the above properties. Then define $f^{N+1}$ as follows: 
given $i\geq i_N\in \N_0$, let $w^{N,i}\in\mc L^N$ be the word such that
$w^{(N,i)}=f^1(x)f^2(x)\ldots f^N(x)$ for $x\in [l^N_{i}, l^N_{i+1})$. 
Let $j = i-i_N$.
If $w^{(N,i)}$ is right special, set
$$
 f^{N+1}(x)=
 \begin{cases}
  a & \text{ if } x\in [l_{2j}^{N+1},l_{2j+1}^{N+1}),\\
  b & \text{ if } x\in [l_{2j+1}^{N+1},l_{2j+2}^{N+1}),
 \end{cases}
$$
otherwise
$$ f^{N+1}(x)=  c \quad\text{ if } x\in [l_{2j}^{N+1},l_{2j+2}^{N+1})=[l_{i}^{N},l_{i+1}^{N}),$$
where $c$ is the unique letter extending $w^{(N,i)}$ to the right. 
This way we have defined $f^{N+1}(x)$ for all $x\geq l^N_{i_N}$. 
We extend it arbitrarily for lower values of $x$. 
It follows from condition (R3) and the assumptions on $f^1,\ldots, f^N$ that 
$x\mapsto f^1(x)f^2(x)\ldots f^N(x)f^{N+1}(x)$ is constant on the intervals 
$[l^{N+1}_{j}, l^{N+1}_{j+1})$ and takes a value in $\mc L^{N+1}$ there.
Furthermore, by induction, any word of $\mc L^{N+1}$ arises in this way; indeed
 $w^{(N,i)}c=f^1(x)f^2(x)\ldots f^N(x)f^{N+1}(x)$ for 
 $x\in [l^{N+1}_{2j}, l^{N+1}_{2j+2})$ where $c$ is the unique extension of $w^{(N,i)}$, if that exists, or $c=a$ ($c=b$) if $w^{(N,i)}$ is right special and 
 $x\in [l^{N+1}_{2j}, l^{N+1}_{2j+1})$ ($x\in [l^{N+1}_{2j+1}, l^{N+1}_{2j+2})$).
\end{proof}

We can take any odometer $(\mc Z,+1)$, but for concreteness we take $\mc Z=\Z_2$. 
Consider a strictly increasing sequence $(t_n)_{n\in \N}$ in $\Z$ such that $t_n \to 0$ in $\Z_2$. For instance, $t_n = 2^{l^n_0}$ is a good choice which we take.
 Its expansion in $\Z_2$ has zeros everywhere except at position $l^n_0$ where we have an entry $1$.
 Define the clopen neighbourhood
$$U_n(z) := U_{\head_{l^n_0}(z)}= \{z'\in\Z_2:\head_{l^n_0}(z) = \head_{l^n_0}(z')\}.$$
It follows from (R1) and our choice of $(t_n)$ that
\begin{align}\label{eq: disjoint supports}
 U_n(t_n)\cap U_{n'}(t_{n'})=\emptyset \text{ for all } n\neq n'\in \N.
\end{align} 
We set $D=\{t_n\:n\in \N\}\cup\{0\}$, a closed set, and define
$f\: \mc Z \setminus D \to \{a,b\}$ with the help of the maps $(f^n)$ from Lemma~\ref{lem-G1} by 
\begin{align}\label{eq: computing semicocycle no 2}
 f(z)=
 \begin{cases}
  f^n(L(z,t_n)) & \text { if } z \in U_n (t_n),\\
  a & \text{ otherwise},
 \end{cases}
\end{align}
where, as before, $L(z,t_n)$ is the length of the common head between $z$ and $t_n$. By \eqref{eq: disjoint supports}, $f$ is well-defined. 
Since all functions $f^n$ are continuous (trivially)
and $L$ is continuous away from the diagonal, $f$ is continuous. 
We fix $\hat z\notin \Z$ (so that $\hat z+\Z\ssq D^c$) and consider the restriction of $f$ to $\hat z+\Z$ an $\{a,b\}$-valued semicocycle over $(\Z_2,+1,\hat z)$. 
As $D$ is not periodic, its stabiliser in $\Z_2$ is trivial and therefore $f$ is separating.
Altogether, $f$ defines a semicocycle extension $(X_f,\sigma)$ of $(\Z_2,+1)$.

\subsection{Tameness or otherwise}
The set of discontinuity points of the semicocyle is $D_{f}=D\subset\Z$, and consequently, all non-integer fibres are regular. The question of whether $(X_f,\sigma)$ is tame or not is therefore a question about its fibre $\pi^{-1}(0)$. We now show that this fibre contains $X_\mathcal L$, the set of all unilateral sequences allowed by the language $\mc L$. 
More precisely, these unilateral sequences are precisely the subsequences of the sequences of $\pi^{-1}(0)$ along the times $(t_n)_n$.
%%%%%%%%%%%%%%%%%%%%%%%%%
\begin{prop}\label{prop-G2}
Consider the dynamical system $(X_f,\sigma)$ from above. 
For any $y\in  X_{\mathcal L}$ there exists
$ x\in \pi^{-1}(0)$ such that 
\begin{equation}\label{eq: defn x-to-the-y}
x_{t_n}=y_n \quad \text{ for all } n\in \N.
\end{equation}
Conversely, for any $ x\in \pi^{-1}(0)$ there is $y\in X_\mathcal L$ such that (\ref{eq: defn x-to-the-y}) holds true. 
\end{prop}
\begin{proof}
Fix $\hat z\in\Z_2\backslash \Z$ and consider the sequence $\hat f \in X_f$ given by 
$\hat f_n = f(\hat z + n)$. 
Given $N\in\N$ and $w\in \mc L^N$, 
pick $i\in\N_0$ such that $w=w^{(N,i)}$ as in Lemma~\ref{lem-G1} (\ref{item: L2}).
Choose $t_w \in \Z$ such that $L(t_w +\hat z,0) = l^N_{i}$ ($t_w$ has the same first $l^N_i$ digits as $-\hat z$ and then disagrees). Then also  $L(t_w +\hat z+t,t) = l^N_{i}$ for all $t\in \Z$. 
Hence, taking $t=t_n$ we deduce with \eqref{eq: computing semicocycle no 2} that
$$\hat f_{t_w+t_n} = f(t_w+ \hat z+t_n)= f^n(L(t_w+\hat z,0))=f^n(l_i^N) = w_n$$
for all $n=1,\ldots,N$.

In other words, the sequence $x=\sigma^{t_w}(\hat f)$ verifies
\begin{align}\label{eq: x along t1 t2...}
 x_{t_1}x_{t_2}\ldots x_{t_N}=f(t_w +\hat z +t_1)f(t_w +\hat z +t_2)\ldots f(t_w +\hat z +t_N)=w.
\end{align}
Now, given $y\in X_\mathcal L$, let $x^y$ be a limit 
point of $\{\sigma^{t_{y_1 y_2\ldots y_N}}(\hat f)\: N\in \N\}$. As
$L(t_{y_1 y_2\ldots y_N}+\hat z,0)\geq l^N_0 \stackrel{N\to +\infty}\longrightarrow +\infty$, 
the sequence $t_{y_1 y_2\ldots y_N}+\hat z$ tends to $0$ and thus $\pi(x^y) = 0$. 
Due to \eqref{eq: x along t1 t2...}, we further have
$ x^y_{t_n}=y_n$ for all $n=1,\cdots,N$ and $N\in \N$, i.e.\ (\ref{eq: defn x-to-the-y}). As $y\in X_\mathcal L$ was arbitrary, this shows the first statement. 
The second statement now follows with Lemma~\ref{lem-G1} 
(\ref{item: defn w Ni}).
\end{proof}
\begin{remark}
 The above idea of embedding a subshift $X_{\mc L}$ in the almost automorphic shift $X_f$ (as in the previous  proposition) is similar in spirit to Williams' classical constructions in \cite{Williams} despite the considerable differences on a technical level.
\end{remark}

\begin{thm}\label{thm: tame in forward but non-tame in general}
 $(X_f,\sigma)$ is backward tame.
 Further $(X_f,\sigma)$ is forward tame if and only if the $\N$-action $(X_{\mc L},\sigma)$ is tame.

\end{thm}
\begin{proof}

To prove the first statement,
let $z\in \Z_2$. 
Its backward orbit $\{z-n : n\in\N_0\}$ can intersect at most finitely many positive integers $\Z^+\subset\Z_2$, hence only finitely many points of $D\subset \Z^+$.
It therefore follows from Corollary~\ref{cor: sufficient criterion for tameness} that if we restrict to the backward dynamics, then $(X_f,\sigma)$ is tame.

Suppose the $\N$-action $(X_{\mathcal L},\sigma)$ is non-tame. 
Then by Remark~\ref{rem: tameness of shifts with finite alphabet},
 there is a sequence of natural numbers $(\nu_n)$ such that for every choice function $\varphi\in \{a,b\}^\N$ there
 is $y^\varphi\in X_\mathcal L$ with ${y^\varphi}_{\nu_n}=\varphi(n)$ for each $n\in \N$.
By Proposition~\ref{prop-G2}, we find for every choice function $\varphi$ an element $x^\varphi\in X_f$ such that  $y^\varphi$ is the subsequence of $x^\varphi$ corresponding to the times $(t_n)_n$. Hence $(t_{\nu_n})_n$ is an independence set of $(X_f,\sigma)$ for the pair of cylinder sets $[a]$ and $[b]$.

Conversely, suppose that $(X_f,\sigma)$ is  forward non-tame. As $D'=\{0\}$, the first part of Theorem~\ref{thm-reduction} implies that 
there is an independence sequence $(\tau_n)$ for $\{a\},\{b\}$ which converges to $z\in\mc Z$ and such that all choice functions are realised by elements of the fibre $\pi^{-1}(-z)$, that is,  
 $\Sigma_{(\tau_n)}(\{0\})=\{a,b\}^\N$. In particular, $\pi^{-1}(-z)$ must be infinite and hence $z\in \Z$.

Now, Lemma~\ref{lem-finite} gives that if  
$\{\tau_n-z\:n\in \N\}\cap D$ was finite, then $\Sigma_{(\tau_n)}(\{z-z\})$ was finite, which is a contradiction.
Therefore, $\{\tau_n-z\:n\in \N\}\cap D$ is infinite. 
Thus there is a subsequence $(t_{n_j})_j$ such that $\{(t_{n_j})_j\}=\{\tau_n-z\:n\in \N\}\cap D$. 
Together with the second part of Proposition~\ref{prop-G2}, we see that 
$(n_j)$ is an independence set for $(X_{\mathcal L},\sigma)$.
\end{proof}

\bibliography{tameness-in-semi-cocycles}{}

\providecommand{\bysame}{\leavevmode\hbox to3em{\hrulefill}\thinspace}
\providecommand{\MR}{\relax\ifhmode\unskip\space\fi MR }
\providecommand{\MRhref}[2]{%
  \href{http://www.ams.org/mathscinet-getitem?mr=#1}{#2}
}
\providecommand{\href}[2]{#2}
\begin{thebibliography}{10}

\bibitem{Auslander1988}
J.~Auslander, \emph{Minimal flows and their extensions}, North-Holland
  Mathematics Studies, vol. 153, North-Holland Publishing Co., Amsterdam, 1988,
  Notas de Matem\'atica [Mathematical Notes], 122. \MR{956049}

\bibitem{BaakeJaegerLenz2016}
M.~Baake, T.~J{\"a}ger, and D.~Lenz, \emph{Toeplitz flows and model sets},
  Bull.\ Lond.\ Math.\ Soc. \textbf{48} (2016), no.~4, 691--698.

\bibitem{BulatekKwiatowski1992}
Wojciech Bu\l{}atek and Jan Kwiatkowski, \emph{Strictly ergodic {T}oeplitz
  flows with positive entropies and trivial centralizers}, Studia Math.
  \textbf{103} (1992), no.~2, 133--142. \MR{1199322}

\bibitem{CovenQuasYassawi2016}
Ethan~M. Coven, Anthony Quas, and Reem Yassawi, \emph{Computing automorphism
  groups of shifts using atypical equivalence classes}, Discrete Anal. (2016),
  Paper No. 3, 28. \MR{3533302}

\bibitem{deVries1993}
J.~de~Vries, \emph{Elements of topological dynamics}, Mathematics and its
  Applications, vol. 257, Kluwer Academic Publishers Group, Dordrecht, 1993.
  \MR{1249063}

\bibitem{Dekking1977}
F.~M. Dekking, \emph{The spectrum of dynamical systems arising from
  substitutions of constant length}, Z. Wahrscheinlichkeitstheorie und Verw.
  Gebiete \textbf{41} (1977/78), no.~3, 221--239. \MR{461470}

\bibitem{DownarowiczDurand2002}
T.~Downarowicz and F.~Durand, \emph{Factors of {T}oeplitz flows and other
  almost {$1-1$} extensions over group rotations}, Math. Scand. \textbf{90}
  (2002), no.~1, 57--72. \MR{1887094}

\bibitem{DownarowiczKasjan2015}
T.~Downarowicz and S.~Kasjan, \emph{Odometers and {T}oeplitz systems revisited
  in the context of {S}arnak's conjecture}, Studia Math. \textbf{229} (2015),
  no.~1, 45--72. \MR{3459905}

\bibitem{DownarowiczLacroix1998}
T.~Downarowicz and Y.~Lacroix, \emph{Almost {$1$}-{$1$} extensions of
  {F}urstenberg-{W}eiss type and applications to {T}oeplitz flows}, Studia
  Math. \textbf{130} (1998), no.~2, 149--170. \MR{1623344}

\bibitem{Downarowicz1991}
Tomasz Downarowicz, \emph{The {C}hoquet simplex of invariant measures for
  minimal flows}, Israel J. Math. \textbf{74} (1991), no.~2-3, 241--256.
  \MR{1135237}

\bibitem{Downarowicz2005}
\bysame, \emph{Survey of odometers and {T}oeplitz flows}, Algebraic and
  topological dynamics, Contemp. Math., vol. 385, Amer. Math. Soc., Providence,
  RI, 2005, pp.~7--37. \MR{2180227}

\bibitem{DownarowiczMaass2008}
Tomasz Downarowicz and Alejandro Maass, \emph{Finite-rank {B}ratteli-{V}ershik
  diagrams are expansive}, Ergodic Theory Dynam. Systems \textbf{28} (2008),
  no.~3, 739--747. \MR{2422014}

\bibitem{DownarowiczSerafin2003}
Tomasz Downarowicz and Jacek Serafin, \emph{Possible entropy functions}, Israel
  J. Math. \textbf{135} (2003), 221--250. \MR{1997045}

\bibitem{DurandHostSkau1999}
F.~Durand, B.~Host, and C.~Skau, \emph{Substitutional dynamical systems,
  {B}ratteli diagrams and dimension groups}, Ergodic Theory Dynam. Systems
  \textbf{19} (1999), no.~4, 953--993. \MR{1709427}

\bibitem{Forrest1997}
A.~H. Forrest, \emph{{$K$}-groups associated with substitution minimal
  systems}, Israel J. Math. \textbf{98} (1997), 101--139. \MR{1459849}

\bibitem{forrest2002cohomology}
Alan~H Forrest, John~R Hunton, and Johannes Kellendonk, \emph{Cohomology of
  canonical projection tilings}, Communications in mathematical physics
  \textbf{226} (2002), no.~2, 289--322.

\bibitem{FuhrmannGroegerJaeger2016}
G.~Fuhrmann, M.~Gr\"{o}ger, and T.~J\"{a}ger, \emph{Amorphic complexity},
  Nonlinearity \textbf{29} (2016), no.~2, 528--565. \MR{3461608}

\bibitem{FuhrmannGlasnerJagerOertel2018}
Gabriel {Fuhrmann}, Eli {Glasner}, Tobias {J{\"a}ger}, and Christian {Oertel},
  \emph{{Irregular model sets and tame dynamics}}, arXiv e-prints (2018),
  1--22, arXiv:1811.06283.

\bibitem{FuhrmannGroger2019}
Gabriel Fuhrmann and Maik Gr{\"o}ger, \emph{Constant length substitutions,
  iterated function systems and amorphic complexity}, Math. Z. (2019).

\bibitem{FuhrmannKwietniak2020}
Gabriel Fuhrmann and Dominik Kwietniak, \emph{On tameness of almost automorphic
  dynamical systems for general groups}, Bulletin of the London Mathematical
  Society \textbf{52} (2020), no.~1, 24--42.

\bibitem{GjerdeJohansen}
Richard Gjerde and {\O}rjan Johansen, \emph{Bratteli-{V}ershik models for
  {C}antor minimal systems: applications to {T}oeplitz flows}, Ergodic Theory
  Dynam. Systems \textbf{20} (2000), no.~6, 1687--1710. \MR{1804953}

\bibitem{Glasner2006}
E.~Glasner, \emph{On tame dynamical systems}, Colloq. Math. \textbf{105}
  (2006), no.~2, 283--295. \MR{2237913}

\bibitem{Glasner2018}
\bysame, \emph{The structure of tame minimal dynamical systems for general
  groups}, Invent. Math. \textbf{211} (2018), no.~1, 213--244. \MR{3742758}

\bibitem{GlasnerMegrelishvili2006}
E.~Glasner and M.~Megrelishvili, \emph{Hereditarily non-sensitive dynamical
  systems and linear representations}, Colloq. Math. \textbf{104} (2006),
  no.~2, 223--283. \MR{2197078}

\bibitem{GlasnerMegrelishvili2018}
\bysame, \emph{More on tame dynamical systems}, Ergodic theory and dynamical
  systems in their interactions with arithmetics and combinatorics, Lecture
  Notes in Math., vol. 2213, Springer, Cham, 2018, pp.~351--392. \MR{3821724}

\bibitem{GottschalkHedlund1955}
Walter~Helbig Gottschalk and Gustav~Arnold Hedlund, \emph{Topological
  dynamics}, American Mathematical Society Colloquium Publications, Vol. 36,
  American Mathematical Society, Providence, R. I., 1955. \MR{0074810}

\bibitem{HermanPutnamSkau}
Richard~H. Herman, Ian~F. Putnam, and Christian~F. Skau, \emph{Ordered
  {B}ratteli diagrams, dimension groups and topological dynamics}, Internat. J.
  Math. \textbf{3} (1992), no.~6, 827--864. \MR{1194074}

\bibitem{Huang2006}
W.~Huang, \emph{Tame systems and scrambled pairs under an abelian group
  action}, Ergodic Theory Dynam. Systems \textbf{26} (2006), no.~5, 1549--1567.
  \MR{2266373}

\bibitem{Iwanik1996}
A.~Iwanik, \emph{Toeplitz flows with pure point spectrum}, Studia Math.
  \textbf{118} (1996), no.~1, 27--35. \MR{1373622}

\bibitem{Jacobs-Keane}
Konrad Jacobs and Michael Keane, \emph{{$0-1$}-sequences of {T}oeplitz type},
  Z. Wahrscheinlichkeitstheorie und Verw. Gebiete \textbf{13} (1969), 123--131.
  \MR{255766}

\bibitem{KerrLi2007}
D.~Kerr and H.~Li, \emph{Independence in topological and {$C^*$}-dynamics},
  Math. Ann. \textbf{338} (2007), no.~4, 869--926. \MR{2317754}

\bibitem{Kohler1995}
A.~K\"{o}hler, \emph{Enveloping semigroups for flows}, Proc. Roy. Irish Acad.
  Sect. A \textbf{95} (1995), no.~2, 179--191. \MR{1660377}

\bibitem{KohmotoOono1984}
M.~Kohmoto and Oono Y., \emph{{Cantor spectrum for an almost periodic
  Schr\"odinger operator and a dynamical map}}, Phys. Lett. A \textbf{102}
  (1984), 145--148.

\bibitem{MarkleyPaul1979}
Nelson~G. Markley and Michael~E. Paul, \emph{Almost automorphic symbolic
  minimal sets without unique ergodicity}, Israel J. Math. \textbf{34} (1979),
  no.~3, 259--272 (1980). \MR{570885}

\bibitem{Oxtoby:1952}
John~C. Oxtoby, \emph{Ergodic sets}, Bull. Amer. Math. Soc. \textbf{58} (1952),
  116--136. \MR{47262}

\bibitem{Paul1976}
M.E. Paul, \emph{Construction of almost automorphic symbolic minimal flows},
  General Topology and Appl. \textbf{6} (1976), no.~1, 45--56. \MR{0388365}

\bibitem{Toeplitz:1928}
O.~Toeplitz, \emph{Ein {B}eispiel zur {T}heorie der fastperiodischen
  {F}unktionen}, Math. Ann. \textbf{98} (1928), no.~1, 281--295. \MR{1512405}

\bibitem{Veech1965}
W.~A. Veech, \emph{Almost automorphic functions on groups}, Amer. J. Math.
  \textbf{87} (1965), 719--751. \MR{0187014}

\bibitem{Livshits-Vershik}
A.~M. Vershik and A.~N. Livshits, \emph{Adic models of ergodic transformations,
  spectral theory, substitutions, and related topics}, Representation theory
  and dynamical systems, Adv. Soviet Math., vol.~9, Amer. Math. Soc.,
  Providence, RI, 1992, pp.~185--204. \MR{1166202}

\bibitem{Williams}
Susan Williams, \emph{Toeplitz minimal flows which are not uniquely ergodic},
  Z. Wahrsch. Verw. Gebiete \textbf{67} (1984), no.~1, 95--107. \MR{756807}

\end{thebibliography}
\bibliographystyle{amsplain}

\end{document}